\documentclass{amsart}
\usepackage{amsthm,amstext,amsmath,amscd,amssymb,latexsym}
\usepackage{color}
\usepackage[matrix,arrow]{xy}
\usepackage{amsfonts,enumerate,array}
\usepackage{todonotes}
\usepackage{verbatim}
\usepackage{amssymb} 

\usepackage{booktabs}

\usepackage{hyperref}
\hypersetup{
    colorlinks=true,         
    linkcolor=red,          
    citecolor=blue,          
    filecolor=blue,          
    urlcolor=blue            
}

\newcommand{\PP}{{\mathbb P}}

\newcommand{\LL}{{\mathcal L}}

\newcommand{\II}{{\mathcal I}}
\newcommand{\JJ}{{\mathcal J}}

\newcommand{\Bs}{{\rm{Bs}}}

\newcommand{\ls}{{\mathcal{L}}}

\newcommand{\Cr}{{\rm{Cr}}}
\newcommand{\N}{{\rm{N}}}
\newcommand{\Eff}{{\rm{Eff}}}

\DeclareMathOperator{\vdim}{vdim}
\DeclareMathOperator{\edim}{edim}
\DeclareMathOperator{\ldim}{ldim}

\DeclareMathOperator{\Pic}{Pic}

\DeclareMathOperator{\NE}{NE}

\newcommand{\paper}{: \begin{it}}
\newcommand{\jour }{, \end{it}}

\newtheorem{theorem}{Theorem}[section]
\newtheorem{lemma}[theorem]{Lemma}
\newtheorem{proposition}[theorem]{Proposition}
\newtheorem{corollary}[theorem]{Corollary}

\newtheorem{question}[theorem]{Question}
\newtheorem{definition}[theorem]{Definition}

\newtheorem{example}[theorem]{Example}

\theoremstyle{remark}
\newtheorem{remark}[theorem]{Remark}

\numberwithin{equation}{section}

\begin{document}

\title{Vanishing theorems for linearly obstructed divisors}

\author{Olivia Dumitrescu}
\address{
Olivia Dumitrescu:
Central Michigan University\\
Pearce Hall 209\\
Mount Pleasant, MI 48859\\}

\address{and Simion Stoilow Institute of Mathematics\\
Romanian Academy\\
21 Calea Grivitei Street\\
010702 Bucharest, Romania}
\email{dumit1om@cmich.edu}

\thanks{The research of O. D.\ has been supported by a grant from the Max-Planck Institute for Mathematics, Bonn, the Arthur J.~Krener fund in the University of 
California, Davis, and GRK 1463 \emph{Analysis,
Geometry, and String Theory} at the 
Leibniz Universit\"at 
 Hannover. O. D. is member of ``Simion Stoilow'' Institute of Mathematics 
of the Romanian Academy (http://www.imar.ro/)}

\author{Elisa Postinghel}
\address{Elisa Postinghel: Department of Mathematical Sciences, Loughborough University, LE11 3TU, UK}
\email{{\tt E.Postinghel@lboro.ac.uk}}

\thanks{The research of E. P. was partially supported by the project \emph{Secant varieties, computational complexity
and toric degenerations} realised within the Homing Plus programme of Foundation for Polish Science,
co-financed from EU, Regional Development Fund and by the Research Foundation - Flanders (FWO)}





\maketitle

\centerline{\emph{Dedicated to the memory of Anthony Geramita}}

\begin{abstract}
We study  divisors on the blow-up of $\PP^n$ at  points in general position
that are non-special with respect to the notion of linear speciality introduced
 in \cite{BraDumPos}.
We describe the cohomology groups 
of their strict transforms  via the blow-up of the space along their linear base locus.
We extend the result to 
non-effective divisors that sit in a small region outside the effective cone.
As an application, we describe linear systems of divisors in $\PP^n$ blown-up at
points in star configuration and their strict transforms via the blow-up of the linear base locus. 
\end{abstract}

\section{Introduction}

The motivation for studying vanishing theorems of divisors comes from Birational Geometry (Mori's Minimal Model Program, see \cite{KM}) and Commutative Algebra (higher order embeddings of projective varieties, see \cite{bfs}). In particular, vanishing theorems have applications to positivity properties of divisors such as global generation, very ampleness and, more in general, $k$-very ampleness properties.

 We denote by $\LL=\LL_{n,d}(m_1,\ldots,m_s)$  the linear system of 
hypersurfaces of degree $d$ in $\PP^n$ passing through a collection of $s$ points in general position with multiplicities at least
$m_1,\ldots,m_s\ge0$ respectively.
The  {\em (affine) virtual dimension} of $\LL$ is denoted by
$$\vdim(\LL)=\binom{n+d}{n}-\sum_{i=1}^s\binom{n+m_i-1}{n},$$
and the {\em expected dimension} of $\LL$ is defined to be $\edim(\LL)=\max(\vdim(\LL),0)$. 
The problem of computing the dimension of such linear systems is often referred to  as \emph{(polynomial) interpolation problem in} $\PP^n$ (see e.g. \cite{Ciliberto} for an account).

If $D$ is the strict transform of a general divisor in $\LL$ in the blow-up $X$ of $\PP^n$ at the $s$ points, 
\begin{equation}\label{general divisor}
 D:= dH-\sum_{i=1}^s m_i E_i,
\end{equation}
 then $\vdim(D):=\vdim(\LL)$ equals $\chi(X, \mathcal{O}_X(D))$, the Euler characteristic of the sheaf on $X$ associated with $D$,  while $\dim(\ls)$ is the number of global section of $\mathcal{O}_X(D)$, namely the dimension of the space $H^0(X, \mathcal{O}_X(D))$.
Using the terminology of the interpolation problem, we will  refer to $D$ as a divisor of degree $d$ \emph{interpolating} $s$ general points with assigned multiplicities $m_1,\dots,m_s$.  
 
The inequality
$\dim(\LL)\ge\edim(\LL)$ is  always satisfied.
However, if the conditions imposed by the assigned multiple points are not linearly independent, 
then the actual dimension of $\LL$ is strictly  greater than the expected one: 
in that case we say that $\LL$  (or $D$) is \emph{special}.
 Otherwise, whenever the actual and the expected dimension coincide 
we say that $\LL$ is \emph{non-special}. The \emph{speciality} of $\LL$ (or $D)$ is defined to be the difference $\dim(\ls)-\edim(\ls)$.

In the last century the problem of computing the dimension (or, equivalently, computing the speciality) of  linear systems  was studied with different techniques by many people. 
In the planar case, the Segre-Harbourne-Gimigliano-Hirschowitz conjectures predicts all  special  linear systems. 
This famous conjecture gives information about the Mori cone of $X$, $\overline{\NE}(X)$, together with its dual, the nef cone of $X$. For example, one of its implications is the so called $(-1)$-Curves Conjecture, see   \cite[Conjecture 3.2.1]{CHMR} and \cite[Conjecture 1.1]{deF}. 
This  consists of a geometric description of the Mori cone of $X$: while the $K$-negative part of  $\overline{\NE}(X)$, namely  the set of classes intersecting negatively the canonical divisor, is known to be generated by classes of $(-1)$-curves,  in the case $s\ge 10$ the $K$-non-negative part would be  a
region with circular portions of  boundary.

The degeneration technique introduced by Ciliberto and Miranda
 (see e.g.\ \cite{Ciliberto,CHMR}) is a successful method in the study of interpolation problems.
However, in spite of many partial results, both
conjectures are still open in general.

In the case of $\PP^3$, there is an analogous conjectural classification 
of special linear systems formulated by Laface and Ugaglia (see e.g. \cite{laface-ugaglia-TAMS}).

Due to its complexity and mysterious geometry,
 the simple question of predicting and computing dimensions of such vector spaces is not even
 conjectured when $n$ is four or higher.
In the case of $\PP^n$ general results are rare and few things are known. 
The well-known Alexander-Hirschowitz Theorem states that a linear system in
 $\PP^n$ with an arbitrary number of double points in general position is non-special except in a list of 
exceptional cases in small degree (see e.g. \cite{AH3,ale-hirsch,ale-hirsch2} for more details). 
For higher multiplicities, the only  general result known so far is  a complete cohomological classification of the speciality of {\em only
 linearly obstructed} effective divisors, proved by Brambilla and the two authors of this manuscript in  \cite{BraDumPos} (see also \cite{Chandler}).
One of the goals of this paper is to extend such a classification  to the non-effective case. 
 
In order to classify the special divisors, one has to understand first what are the \emph{obstructions}, 
namely what are the varieties that whenever contained with multiplicity in the base locus of a given
 divisor force $\LL$ to be special. In \cite{boccin,bocci2} these varieties are named {\em special effect varieties}. 
Few examples of obstructions were classified 
before \cite{BraDumPos}. The only  examples known were $(-1)$-curves in
 $\PP^2$ and $\PP^3$ (see \cite{laface-ugaglia-TAMS}) and those appearing in the list of exceptions from the Alexander-Hirschowitz Theorem. Theorem \ref{monster for effective} below (that was proved in \cite[Theorem 4.6]{BraDumPos}) and  
Corollary \ref{number of sections} show that, for any effective divisor, {\em linear cycles} of arbitrary dimension are always obstructions.

  In the direction of extending the Segre-Harbourne-Gimigliano-Hirschowitz Conjecture to $\PP^n$ and possibly to 
other rational projective varieties, we pose  the same natural and general question as in \cite{BraDumPos}.

\begin{question}[{\cite[Question 1.1]{BraDumPos}}]\label{question}

Consider any non-empty linear system $\LL$ in $\PP^{n}$. 
Let $\widetilde{\mathcal{X}}$ be the smooth composition of blow-ups of $\PP^n$ along the 
(strict transforms of the) 
cycles of the base locus of $\ls$, ordered in increasing dimension. We denote by $D$ a general divisor of the linear system  $\LL$, and by $\widetilde{{\mathcal{D}}}$ the strict transform of $D$  in $\widetilde{\mathcal{X}}$.
Does $h^i(\widetilde{\mathcal{X}},\mathcal{O}_{\widetilde{\mathcal{X}}}(\widetilde{\mathcal{D}}))$ vanish for all $i\ge 1$?

\end{question}
We remark that $\widetilde{\mathcal{D}}$ is obtained by suitably blowing-up the whole base locus of $D$ and subtracting the fixed hypersurfaces. 
Precisely, since these divisorial components of the base locus split off the system, taking proper transform under blow-up is equivalent to the deletion of these divisorial components.

For linear divisorial components more details are presented in  Section \ref{blow-up construction}.

In general Question \ref{question} is difficult to answer since it requires first to describe 
the base locus of a linear system and second to compute the 
cohomology of the strict transform after blowing it up.

An affirmative answer to Question \ref{question} implies that 
$\dim(\LL)=\chi(\widetilde{\mathcal{X}}, \mathcal{O}_{\widetilde{\mathcal{X}}}(\widetilde{\mathcal{D}}))$, translating the classical dimensionality problem for linear systems
 into a Riemann-Roch formula for divisors living in subsequently blown-up spaces.

We denote by $\tilde{D}$ the strict transform of $D$ in $\tilde{X}$, the
 blow-up of $X$ along the linear cycles of the base locus of $D$. A precise definition  is given below in \eqref{proper transform 1} with $r=n-1$. We remark that $\tilde{D}$ is different from $\widetilde{\mathcal{D}}$ that is introduced in Question \ref{question}:  
the second one denotes the strict transform of $D$ in $\widetilde{\mathcal{X}}$,  the blow-up of $X$ along all -linear and non linear - cycles of the base locus of $D$.

Due to the combinatorial and geometrical 
complexity of this problem so far we only understand properties of divisors $\tilde{D}$ and we will present them in detail. 

We  mention that for $s\leq n+2$ the divisors $\tilde{D}$ obtained by subsequent blow-up of the linear base locus (described in details in Sections \ref{section non-effective statements} and \ref{section non-effective proofs}) are divisors in $\overline{\mathcal{M}}_{0,n+3}$, the moduli space of stable rational curves with $n+3$ marked points, see \cite{Ka}. Therefore understanding their cohomological description can be used in the study of positivity properties such as the 
effective cone and the ample cone of $\overline{\mathcal{M}}_{0,n+3}$ (see also {\cite[Section 6.3]{BraDumPos}}).

In the article \cite{BraDumPos}, the authors introduced a new notion 
of expected dimension for linear systems, 
that takes into account the linear obstructions and extends the notion of virtual dimension, 
namely the {\em linear virtual dimension}. In this paper, we will use $\ldim(\LL)$ to denote the (affine) linear virtual dimension, 
instead of the (projective) expected linear dimension as used in \cite{BraDumPos}.
Given two linear systems $\ls_{n,d}(m_1,\dots,m_s)$ and $\ls_{n,d}(m'_1,\dots,m'_s)$
 with the same degree,
we write $\ls\prec_s\ls'$  if $m_i\ge m'_i$ for 
all $i\in\{1,\dots,s\}$.

\begin{definition}[{\cite[Definition 3.2]{BraDumPos}}]\label{new-definition}
Given a linear system $\ls=\ls_{n,d}(m_1,\dots,m_s)$, for any integer 
$-1\le r\le  s-1$ and for any multi-index $I(r)=\{i_1,\ldots,i_{r+1}\}\subseteq\{1,\ldots,s\}$, define the integer 
\begin{equation}\label{mult k}
k_{I(r)}:=\max(m_{i_1}+\cdots+m_{i_{r+1}}-rd,0).
\end{equation}

The {\em (affine) linear virtual dimension} of $\LL$ (or of $D$), denoted by $\ldim(\LL)$,   is the number
\begin{equation}\label{linvirtdim}
\sum_{r=-1}^{s-1}\sum_{I(r)\subseteq \{1,\ldots,s\}} (-1)^{r+1}\binom{n+k_{I(r)}-r-1}{n},
\end{equation}
 where we set $I(-1)=\emptyset$. 
The {\em (affine) linear expected dimension} of $\LL$
is defined as follows: it is $0$ if $\LL\prec_s\ls'$ and $\ldim(\ls')\le0$,
otherwise it is the maximum between  $\ldim(\LL)$ and $0$.
\end{definition}
 We remark that  this notion is well-defined not only for all effective divisors but also for non-effective ones. We will study this type of divisors in Sections \ref{section non-effective statements} and \ref{section non-effective proofs}.

In this light, asking whether the dimension of a
 given linear system equals its linear expected dimension can be thought of
as a refinement of the classical question of asking whether the dimension 
equals the expected dimension. If the answer to this question is affirmative, then $\LL$ (or $D$) 
is said to be a  {\em only linearly obstructed}. Obviously, non-special linear systems are always only linearly obstructed.

There exist linear systems that are {\em linearly obstructed} without being 
{\em only linearly obstructed}. For instance $\LL_{4,10}(6^7)$ contains all lines 
$L_{ij}$, $i,j\in\{1,\dots,7\}$  with multiplicity two in its base locus as well as
 the rational normal curve through the seven points, see \cite[Example 6.2]{BraDumPos} for more details.

Connections between $\LL$ being only linearly obstructed
and the Fr\"oberg-Iarrobino Conjecture (see \cite{Chandler}),
describing the Hilbert series of an ideal generated by $s$
forms, can be found in {\cite[Section 6]{BraDumPos}}. 
This reveals the importance of the notion of linear speciality, that was achieved and developed independently from both the geometric and the algebraic setting.

 Linear systems with an arbitrary number of points and with bounded sum of the multiplicities were classified in \cite{BraDumPos}, for $n\geq1$, $d\geq 2$, by proving that they are only linearly obstructed. 

\begin{theorem}[{\cite[Theorem 5.3]{BraDumPos}}]\label{theorem a n+3}
All non-empty linear systems of the form $\ls=\ls_{n,d}(m_1,\dots,m_s)$ with $s\le n+2$ base points  are only linearly obstructed. 
Moreover, if $s\ge n+3$ and
\begin{flushleft}
\eqref{EffectivityCondition}\quad\quad\quad\quad\quad
$\sum_{i=1}^s m_i\leq nd+ \min (n-s(d),s-n-2), \quad 1\le m_i\le d$,
\end{flushleft}
where $s(d)\geq0$ is the number of points of multiplicity $d$,
then $\LL$ is non-empty and only linearly obstructed.
\end{theorem}

 The new perspective  introduced in \cite{BraDumPos} is built upon the cohomological study 
of the strict transforms of effective and only
 linearly obstructed divisors. More precisely, the strict transforms are taken
 after subsequently blowing-up their linear base locus, first the lines, then the planes, etc.
 Moreover, in \cite{BraDumPos} a complete classification was given for effective divisors 
 interpolating $s\le n+2$ general points with assigned multiplicities,  in which range 
the effective cone was known (see for example \cite{CDD}). 

For every effective divisor $D$,  let $D_{(r)}$ denote the strict transform of $D$ 
in the space $X^n_{(r)}$ obtained as a sequence of blow-ups of $\PP^n$ along the linear base locus 
of $D$  up to dimension $r$, with $r\leq n-1$ (we refer to Section \ref{blow-up construction} for details about this construction):
\begin{equation}\label{proper transform 1}
D_{(r)}:=D-\sum_{\rho=1}^{r}\sum_{I(\rho)\subseteq\{1,\dots,s\}}k_{I(\rho)}E_{I(\rho)},
\end{equation}
 where $E_{I(\rho)}$ denotes the (strict transform of the) exceptional divisor of the 
linear subspace of $\PP^n$ of dimension $\rho$
spanned by the points parametrised by the multi-index $I(\rho)$ and $k_{I(\rho)}$ 
is the non-negative  integer introduced in \eqref{mult k}, which is the multiplicity with 
which the aforementioned subspace is contained in the base locus; this will be proved in Section \ref{linear base locus lemma}.
Let $\bar{r}$ be the maximum dimension of the linear base locus;
we will set 
\begin{equation}\label{proper transform 2}
\tilde{D}:=D_{(\bar{r})}.
\end{equation}
To simplify notation here and throughout the paper we will also abbreviate 
$h^i(X^n_{(r)},\mathcal{O}_{X^n_{(r)}}(D_{(r)}))$
by $h^i(D_{(r)})$.

\begin{theorem}[{\cite[Theorem 4.6]{BraDumPos}}]\label{monster for effective} 
Given integers $d,m_1,\dots,m_s$, consider the divisor \eqref{general divisor}.
If  $s\leq n+2$ and $D$ is effective, the following statements hold.
\begin{enumerate}
\item[(a)]
$h^0(D)=\ldim(D)$ and  $h^{i}(\tilde{D})=0$ for every $i\geq 1$. 
\item[(b)] For any $0\leq r\leq n-1$, $h^i(D_{(r)})=0$ for every $i\geq 1$ and $i\ne  r+1$, while
$$
h^{r+1}(D_{(r)})=\sum_{\rho=r+1}^{s-1}\sum_{I(\rho)\subseteq\{1,\dots,s\}}
 (-1)^{\rho-r-1}{{n+k_{I(\rho)}-\rho-1}\choose{n}}.
$$
\end{enumerate}

\end{theorem}

The goal of this paper is to  show that the same type of results 
as in Theorem \ref{monster for effective}  holds for larger classes of divisors, 
such as effective divisors with arbitrary number of general base points and non-effective divisors.
The definition of strict transform after blowing-up the linear base locus is formally extended to the non-effective case 
in Section \ref{section base locus empty case}.

We also extend the formula in Theorem \ref{monster for effective} to the case 
of any effective divisor, not necessarily only linearly obstructed,
 interpolating an arbitrary collection
of general multiple points.

\begin{theorem}\label{higher cohom intro}
Given integers $d,m_1,\dots,m_s$, consider the divisor $D$ of the form  \eqref{general divisor}.
If $D$ is effective, then
 for any $0\leq r\leq n-1$ we have
\begin{align*}
h^{r+1}(D_{(r)})=&\sum_{\rho=r+1}^{s-1}\sum_{I(\rho)\subseteq\{1,\dots,s\}}
 (-1)^{\rho-r-1}{{n+k_{I(\rho)}-\rho-1}\choose{n}}\\
& +\sum_{\rho=r+1}^{n}(-1)^{\rho-r-1}h^{\rho}(\tilde{D}).
\end{align*}
In particular, $$h^0(D)=\ldim(D)+\sum_{\rho=1}^{n}(-1)^{\rho+1}h^{\rho}(\tilde{D}).$$

Moreover, if $h^i(\tilde{D})=0$, for all $i\ge 1$, then $h^i(D_{(r)})=0$ for all $i\ne r+1$.
\end{theorem}

This  result is part of Theorem \ref{higher cohomologies} that will be proved in Section
\ref{section non-effective proofs} in a more general setting. 
The geometric interpretation is that for any effective divisor $D$,
 every $\rho$-dimensional linear 
cycle $L_{I(\rho)}$, for which $k_{I(\rho)}\ge 1$ and $\rho\ge r+1$,
gives a contribution with  sign, $(-1)^{r-\rho+1}$,  equal to
 \begin{equation}\label{The Newton binomial}
{{n+k_{I(\rho)}-\rho-1}\choose{n}}
\end{equation}
to $h^{r+1}(D_{(r)})$ and  to the formula for 
$\ldim(D)$ (cfr. Theorem \ref{higher cohomologies}  and Corollary \ref{number of sections}). 
Moreover, such a contribution is zero when $k_{I(\rho)}\le \rho$.

 The main result of this paper is a complete 
cohomological description of $D_{(r)}$ in the following cases,
where we set \begin{equation}\label{b}
b:=b(D)=\sum_{i=1}^sm_i-nd.
\end{equation}

\begin{theorem}\label{general statement vanishings}
Fix $d,m_1,\dots,m_s$.
Statements $(a)$ and $(b)$ of Theorem \ref{monster for effective} hold for 
all  divisors $D$ of the form \eqref{general divisor} with $m_i\le d+1$  under the following hypothesis:
$s\le n+1$  and $b\leq n$, or $s= n+2$ and   $b\le1$, or $s\ge n+3$ and $b\leq\min(n-s(d),s-n-2)$.

Moreover, if $s\le n+1$ then $h^i(\tilde{D})=0$, for all $i\leq n-1$, and 
$h^{n}(\tilde{D})={{b-1}\choose{n}}$ for $b\geq n+1$ and zero otherwise.
\end{theorem}

The theorem summarises the results contained in Theorem \ref{vanishing for >=n+3},
 Theorem \ref{toric theorem},   Theorem  \ref{monster for empty} and
Theorem \ref{empty arbitrary points}.
This result shows that it makes sense to extend Question \ref{question} to non-effective divisors
 in a small region outside the effective cone with a correct definition of $D_{(r)}$.
In order to study classical interpolation problems, a crucial step is the
 study of non-effective divisors. More precisely, whenever a linear cycle
 is contained in the base locus of a divisor $D$ , 
 the normal bundle of its exceptional divisor after blow-up 
is given by a non-effective divisor (see Lemma \ref{normal bundle}), and the cohomology groups of its multiples produce a contributions to the speciality of the form \eqref{The Newton binomial}.

From Theorems \ref{monster for effective} and
 \ref{general statement vanishings},
for any effective divisor with $s\leq n+2$ 
and any only linearly obstructed divisor satisfying the bound 
\eqref{EffectivityCondition}, one obtains $\chi(\tilde{X}, \mathcal{O}_{\tilde{X}}(\tilde{D}))=\ldim(D)$. 
This gives a strong interpretation of the notion of
linear expected dimension that, not only represents a dimension count for the linear system $\LL$, 
 but also computes the Euler characteristic of the sheaf $\mathcal{O}_{\tilde{X}}(\tilde{D})$.
As a corollary of Theorem \ref{general statement vanishings}
we extend this Riemann-Roch formula to  larger classes of divisors obtaining interesting combinatorial identities. 
In particular, toric divisors sitting on the facets of the effective cone have Euler characteristic equal to one.

\medskip

This paper is organised as follows.
In Section \ref{blow-up construction} we introduce the general construction and notation.

 In Section \ref{Cremona of hyperplane class} we provide a cohomological 
classification of a class of interesting divisors, namely integer multiples of \emph{standard Cremona transformations} of the
hyperplane classes.

In Section \ref{vanishing section n+3} we  give an explicit description of the linear base locus of any divisor $D$, 
 Proposition \ref{sharpBLL}. 
In Theorem \ref{vanishing for >=n+3} we show that linear cycles are the only obstructions
 for divisors with  $s\leq n+2$ or 
satisfying \eqref{EffectivityCondition}. 

In Section \ref{section non-effective statements} we give vanishing theorems for the cohomology groups of divisors $D$ with $s\le n+2$ and with multiplicities bounded above by $d+1$, 
since in this range $\ldim$ is well-defined, see Theorem \ref{toric theorem} and 
Theorem \ref{monster for empty}.
In Theorem \ref{empty arbitrary points} we extend the result to the case of  non-effective divisors with 
$s\ge n+3$ points with multiplicities satisfying the bound \eqref{EffectivityCondition}.

 Section \ref{section non-effective proofs} is dedicated to the proofs of the results stated in
Section \ref{section non-effective statements} .

In Section \ref{star configuration} we use the vanishing theorems from Section \ref{section non-effective statements}
to study linear systems with points in special position.
 Theorem \ref{star}  computes the dimensions of a class of linear systems in $\PP^n$ interpolating \emph{star configurations} of points  with higher multiplicities.

\subsection{Acknowledgments} 
We would like to express our sincere gratitude to the referee for carefully reading this manuscript and for many
constructive comments that substantially helped improve the quality of this article. 

We also would like to thank
Chiara Brambilla, Ciro Ciliberto, Rick Miranda and
 Brian Osserman for useful discussions. 

 We also  thank the organisers of the Workshop on Perspectives and Emerging Topics in Algebra and Combinatorics
-PEAKs 2013 (Austria),  funded by DFG Conference Grant HA4383/6-1
for the hospitality during their stay that promoted and made this collaboration possible.

\section{Blowing-up: construction and notation}\label{blow-up construction}

In this section we recall the main construction that was partially presented in 
\cite[Sect. 4.1]{BraDumPos}.

Let $\mathcal{I}$ be a set of subsets of $\{1,\dots, s\}$. 
For every integer $0\leq r\le \min(n,s)-1$, we denote by $I(r)=\{i_1,\dots,i_{r+1}\}\in\mathcal{I}$ 
a \emph{multi-index} of length $|I(r)|=r+1$.
Let us also introduce the notation 
\begin{equation}\label{IIj}
\begin{split}
\mathcal{I}(r)&:=\{I(\rho)\in\mathcal{I}: 0\le \rho\le r\};\\ 
\mathcal{I}(r)_j&:=\{I(\rho)\in\mathcal{I}(r)\setminus\II(0): j\in I(\rho)\};\\
\mathcal{I}(r)_{j_1,j_2}&:=\{I(\rho)\in\mathcal{I}(r)\setminus\II(1): j_1,j_2\in I(\rho)\}.
\end{split}
\end{equation}

Let $p_1,\dots,p_s$ be general points in $\PP^n$ and, for every $I(r)\in \mathcal{I}$,
let 
 $L_{I(r)}\cong\PP^r$ denote the  $r$-dimensional linear subspace spanned by the points 
$\{p_j:\ j\in I(r)\}$, which we will refer to as a \emph{linear $r$-cycle}.
Notice that $L_{I(0)}=p_j$ is a point. An arbitrary multi-index will be denoted by $I$ without specifying its cardinality.

We will assume that $\mathcal{I}$ satisfies the following properties:
\begin{enumerate}
\item[(I)] $\{j\}\in\mathcal{I}$, for all $j\in\{1,\dots,s\}$;
\item[(II)] if $I\subset J$ and $J\in\mathcal{I}$, then $I\in\mathcal{I}$;
\item[(III)] if $I,J\in\mathcal{I}$, then $L_{I}\cap L_J=L_{I\cap J}$.
\end{enumerate}
Let $\Lambda=\Lambda(\mathcal{I})\subset\PP^n$ be
the subspace arrangement corresponding to $\mathcal{I}$, i.e.  the (finite) union of the
linear cycles $L_I$ with $I\in\mathcal{I}$.
Let $\bar{r}$ be the largest dimension of a linear cycle in $\Lambda$, 
i.e.\ $\bar{r}=\max_{I\in\mathcal{I}}(|I|)-1$. 
Write $\Lambda=\Lambda_{(1)}+\cdots+\Lambda_{(\bar{r})}$, 
where $\Lambda_{(r)}=\cup_{I(r)\in\mathcal{I}} L_{I(r)}$.

We denote by $\pi_{(0)}:X_{(0)}^n\to\PP^n$ the blow-up of $\PP^n$ at $p_1,\dots,p_s$, with $E_1,\dots,E_s$ exceptional divisors. 
Let us also consider the sequence of blow-up maps
\[
X_{(n-1)}^n\stackrel{\pi_{(n-1)}}{\longrightarrow} 
\cdots \stackrel{\pi_{(2)}}{\longrightarrow}X_{(1)}^n
\stackrel{\pi_{(1)}}{\longrightarrow}X_{(0)}^n,
\]
where $X_{(r)}^n\stackrel{\pi_{(r)}}{\longrightarrow}X_{(r-1)}^n$
is the blow-up of $X_{(r-1)}^n$ along 
the strict transform of $\Lambda_{(r)}\subset\PP^n$,
via $\pi_{(r-1)}\circ\cdots\circ\pi_{(0)}$.  
For any $I(r)\in\mathcal{I}$, we denote by
$E_{I(r)}$ the exceptional divisor of the cycle $L_{I(r)}$ in $X_{(r)}^n$.
Notice that conditions (I), (II) and (III) ensures that, for every $r$, the sum of the 
strict transforms of the exceptional divisors of $\pi_{(r)}$ is simple normal crossing. 
We will write, abusing notation,  $H$ for the strict transform in $X_{(r)}^n$
of $\mathcal{O}_{\PP^n}(1)$ and 
${E}_{I(\rho)}$, for $I(\rho)\in\mathcal{I}(r-1)$, for the strict transform in $X_{(r)}^n$
of the exceptional divisor $E_{I(\rho)}$ in $X^n_{(\rho)}$, respectively. 

\begin{remark}\label{permutohedron}
Notice that condition (III) is obviously satisfied when $s\le n+1$. In particular if $\II$ is the set of al subsets of $\{1,\dots,n+1\}$, then each map 
$\pi_{(r)}$ is the blow-up of the strict transforms of all coordinate $r$-dimensional subspaces of $\PP^n$. These are  toric morphisms, for every $r$. In particular the defining lattice polytope of $X_{(n-2)}$ is the so called \emph{permutohedron} studied by Kapranov in \cite{kapranov}. The space $X_{(n-2)}$ coincides with the Losev-Manin moduli space studied in \cite{LM}.
\end{remark}

\begin{remark}\label{blow up of hypersurface}
Notice that  the map $\pi_{(n-1)}:X_{(n-1)}^n\longrightarrow X_{(n-2)}^n$ is an isomorphism.  In particular, $\Pic(X_{(n-1)}^n)\cong (\pi_{(n-1)})^\ast \Pic(X_{(n-2)}^n)$. 
Thus, in our notation, for every $I(n-1)\in \mathcal{I}$ we have
 \begin{equation}\label{exc of hyperplane}
 E_{I(n-1)}=H-\sum_{\substack{I(\rho)\in\mathcal{I}(n-2),\\
 I(\rho)\subset I(n-1)}}E_{I(\rho)}\in \Pic(X^n_{(n-2)}).
 \end{equation}
 In other terms $E_{I(n-1)}$ denotes the strict transform on $X^n_{(n-2)}$ of the hyperplane 
 \begin{equation}\label{hyperplane}
 H_{I(n-1)}:=H-\sum_{i\in I(n-1)}E_i \in \Pic(X^n_{(0)}).
 \end{equation}
 For this reason we will sometimes  write 
 \begin{equation}\label{hyperplane tilde} 
 E_{I(n-1)}=\tilde{H}_{I(n-1)}\in\Pic (X^n_{(n-2)}).
 \end{equation}
 We  will use the first notation if we want to stress its nature as exceptional divisor, 
whereas we will use the second notation when we want to consider it as the strict 
transform of a hyperplane.

Finally, notice that $E_{I(n-1)}\cong X^{n-1}_{(n-3)}$, the blow-up of $\PP^{n-1}$ 
along all coordinate subspaces, in increasing dimension, i.e. the $(n-1)-$dimensional
toric variety discussed in Remark \ref{permutohedron}.

\end{remark}

The Picard group of $X^n_{(r)}$ is 
$$\textrm{Pic}(X^n_{(r)})=\langle H, E_{I(\rho)}: I(\rho)\in\II(r-1)  \rangle.$$

\begin{remark}\label{blow up preserves h^i}
For $r=1,\dots, n-1$ and $F$ a divisor on $X^n_{(r-1)}$, for any $i\geq 0$, we have
\[
h^i(X^n_{(r)}, (\pi_{(r)})^\ast F)=h^i(X^n_{(r-1)},F).
\]
It follows from Zariski connectedness theorem and by the projection formula (see for instance \cite{hartshorne}
or \cite[Lemma 1.3]{laface-ugaglia-BullBelg} for a more detailed proof.).
\end{remark}

\subsection{The geometry of the divisor $E_j$ in $X^n_{(r)}$}\label{exc point}
Let $r\ge 1$.
 Consider the composition of blow-ups 
$\pi_{(r,0)}:=\pi_{(r)}\circ\cdots\circ\pi_{(1)}:X^n_{(r)}\to X^n_{(0)}$.
By abuse of notation we will denote by $E_{j}$ 
 the strict transform $(\pi_{(r,0)})^{-1}_\ast E_j\in \textrm{Pic}(X^n_{(r)})$ of the exceptional divisor
 $E_{j}\in\textrm{Pic}(X^n_{(0)})$ 
of the point $p_j \in \PP^n$.

For every multi-index $I(\rho)\in\mathcal{I}(r)_j$,
let  $E_{I(\rho)}\in \textrm{Pic}(X^n_{(r)})$  be the strict transform of the exceptional 
divisor in $X^n_{(\rho)}$ of $L_{I(\rho)}$,  and set $e_{I(\rho)|j} := E_{I(\rho)}|_{E_{j}}$. Moreover, let $h$ be the hyperplane class of $E_j$.
\begin{lemma}\label{Pic of E_j}
In the above notation,
a basis for the  Picard group of $E_j$ is given by $h$ and $e_{I(\rho)|j}$, for all $I(\rho)\in\mathcal{I}(r)_j$. 

In particular, we have the isomorphism $\textrm{Pic}(E_j)\cong \textrm{Pic}(X^{n-1}_{(r-1)})$.
\end{lemma}

Rephrasing Lemma \ref{Pic of E_j},  the exceptional divisor  $E_j\in \textrm{Pic}(X^n_{(r)})$ is isomorphic to a blown-up $\PP^{n-1}$ along linear $(\rho-1)$-cycles, $\rho\le r$, spanned by
subsets of a collection of $s-1$ general points.
These $s-1$ points correspond to the lines $L_{I(1)}$, with $I(1)=\{j,l\}$, for all indices
$l\in\{1,\dots,\hat{j},\dots,s\}$. Similarly, the linear $(\rho-1)$-cycles blown-up in $E_j$ correspond to the 
linear $\rho$-cycles $L_{I(\rho)}$ of $\PP^n$ satisfying the condition that $j\in I(\rho)$.

\subsection{The geometry of the divisor $E_{I(\rho)}$ in $X^n_{(r)}$}\label{geometry of exc}
Let $I=I(\rho)\in\II(r)$ be any multi-index.
Notice that if $\rho=0$,  $E_{I(\rho)}$ is the exceptional divisor of a point that was already described in Section \ref{exc point}.
Consider the composition of blow-ups 
$\pi_{(r,\rho)}:=\pi_{(r)}\circ\cdots\circ\pi_{(\rho+1)}:X^n_{(r)}\to X^n_{(\rho)}$. 
By abuse of notation we will denote by $E_{I}$ 
 the strict transform via $\pi_{(r,\rho)}$ in $X^n_{(r)}$ of the exceptional 
divisor $E_{I}\in\textrm{Pic}(X^n_{(\rho)})$ 
of a linear $\rho$-cycle $L_{I}\subset\PP^n$: $E_I\in \textrm{Pic}(X^n_{(r)})$ is a Cartesian product that we are now going to describe.

Consider first the case $\rho=r$. 
We have the following isomorphism: $E_I\cong X^r_{(r-2)} \times \PP^{n-r-1}$; we refer to \cite[Section 4]{BraDumPos} for details.
The Picard group of the first factor is generated by
 $\langle h, e_{I(t)}: I(t)\subset I, I(t)\in\II(r-2)\rangle,$

 where ${E_{I(t)}}{|_{E_I}}=:e_{I(t)}\boxtimes 0$,
while the Picard group of the second factor is generated by the hyperplane class.

Assume now that $0\le \rho < r$.  
Notice first of all that the restriction  ${E_{I(t)}}{|_{E_I}}$ of $X^n_{(r)}$ is zero on both
 factors unless one of the following 
 containment relations is satisfied: $I \subset I(t)$ or $I(t)\subset I$.
We denote by $h_b$ and $h_f$ the hyperplane classes of the two factors respectively. Moreover
 we introduce divisors $e_{I(t)}$ on the first factor and $e_{I(t)|I}$ on the second factor 
according to the following intersection table:

\begin{equation}\label{intersection table}
\begin{split}
{H}{|_{E_{I}}}&=: h_b \boxtimes 0;\\ 
{E_{I(t)}}{|_{E_I}}&=: e_{I(t)}\boxtimes 0, \textrm{ for all } I(t)\subset I,\ t\ge 0;\\ 
{E_{I(t)}}{|_{E_{I}}}&=: 0\boxtimes e_{I(t)|I}, \textrm{ for all } I \subset I(t), \ t\le r.
\end{split}
\end{equation}
Notice that if $\rho=0$, the first factor of $E_{I(\rho)}$ is a point and we have $h_f=h$ in the notation of Section \ref{exc point}.

\begin{remark}\label{exc of hyperplanes}
If $t=\rho-1$, and $I(t)\subset I$, i.e. $L_{I(\rho-1)}\subset\PP^n$ is a 
hyperplane of $L_I\subset \PP^n$, using \eqref{exc of hyperplane} we obtain the following equality:
$$
{E_{I(\rho-1)}}{|_{E_I}}={\left(H-\sum E_{I(\tau)}\right)}{|_{E_I}}=
\left(h-\sum e_{I(\tau)}\right)\boxtimes0,
$$
where the sums range over the multi-indices $I(\tau)\subset I$, 
$I(\tau)\in\II(\rho-2)$. 
Accordingly, $e_{I(\rho-1)}=h_b-\sum e_{I(\tau)}$.
A similar argument holds for divisors on the second factor, when $t=n-1$ and $I\subset I(t)$.
\end{remark}

\begin{lemma}\label{Picard of factors}
In the above notation, assume $\mathcal{I}$ is a set of multi-indices satisfying conditions \emph{(I),(II)} and \emph{(III)}.
We have
$$E_I\cong X^{\rho}_{(\rho-2)} \times X^{n-\rho-1}_{(r-\rho-1)}.$$

Moreover, bases for the Picard groups of the two factors of the product $E_{I}$, for $I\in\mathcal{I}$ of length $\rho+1$,
are given respectively by
\begin{align*} 
\langle h_b,&\ e_{I(t)}:I(t) \subset I, I(t)\in\II(\rho-2)\rangle;\\
\langle h_f,&\ e_{I(t)|I(\rho)}:I \subset I(t),\ I(t)\in\II\setminus\II(\rho) \rangle.
\end{align*}
\end{lemma}

The element $e_{I(t)}$ represents the (strict transform of the) exceptional divisor in $X^{\rho}_{(\rho-2)}$ of 
a linear $t-$cycle of $\PP^\rho$ and $e_{I(t)|I(\rho)}$
represents the  (strict transform of the)  exceptional divisor in $X^{n-\rho-1}_{(r-\rho-1)}$ of
a linear $(t-\rho-1)-$cycle of $\PP^{n-\rho-1}$, spanned  by $|I(t)\setminus I(\rho)|$ points.

\begin{remark}\label{first factor is toric}
In the above notation,  using the assumptions (I),(II) and (III), one can see that there are exactly  $\rho+1$ exceptional divisors of the form $e_{I(0)}$ on the first factor $X^{\rho}_{(\rho-2)}$ of $E_I$. They  can be thought as the exceptional divisors of the $\rho+1$ coordinate points of $\PP^\rho$. In the same way the $e_{I(t)}$'s, for $t\ge1$, represent the (strict transforms of the) exceptional divisors of $t-$dimensional coordinate subspaces. Hence $X^{\rho}_{(\rho-2)}$ is the 
$\rho-$dimensional toric variety discussed in Remark \ref{permutohedron}.
\end{remark}

We now give a characterisation of the normal bundle of the exceptional divisor $E_I$ in the space $X^n_{(r)}$. To this purpose, we introduce the following divisor on $X^{\rho}_{(\rho-2)}$:
\begin{equation}\label{cremona}
\Cr_\rho(h_b):=\rho h_b-\sum_{\substack{I(t)\subset I\\ I(t)\in\II(\rho-2)}}(\rho-t-1)e_{I(t)}.
\end{equation}
We will give a detailed cohomological description of such a divisor in Section \ref{Cremona of hyperplane class}.

\begin{lemma}\label{normal bundle}
In the notation above, we have
$$
-{E_I}{|_{E_I}}=\Cr_\rho(h_b)\boxtimes h_f.
$$
\end{lemma}
\begin{proof}
The proof follows from the computation of the conormal bundle of the first factor of $E_I$: $$N^\ast_{X^{\rho}_{(\rho-2)}|X^n_{(r)}}=\mathcal{O}_{X^{\rho}_{(\rho-2)}}(\Cr_\rho(h_b)).$$ See \cite[Lemma 4.3]{BraDumPos} for details.
\end{proof}

\subsection{Restriction exact sequences}\label{sequences}
Given a divisor $F\in \Pic(X^n_{(r)})$, we will consider three types of restriction exact sequences.

\subsubsection*{Restriction to the exceptional divisor of a point}
In the same notation as above, let $E_1$ be the strict transform in $X^n_{(r)}$ of the exceptional divisor of the point $p_1\in\PP^n$.
We will call \emph{sequence of type (A)} the following restriction sequence:
\begin{equation*}\label{seq A}
 0\longrightarrow F-E_1\longrightarrow  F \longrightarrow
 {F}{|_{E_1}}\longrightarrow 0.
\tag{$A$}
\end{equation*}

\subsubsection*{Restriction to the exceptional divisor of a  linear cycle}
Fix integers $1\le \rho\le r\le n-2$.
In the same notation as above, let $E_{I(\rho)}$ be the strict transform in $X^n_{(r)}$ of the exceptional divisor of the linear subspace of $\PP^n$ spanned by the points parametrised by $I(\rho)$.
We  call \emph{sequence of type (B)} the following restriction sequence:
\begin{equation*}\label{seq B}
 0\longrightarrow F-E_{I(\rho)}\longrightarrow  F \longrightarrow
 {F}{|_{E_{I(\rho)}}}\longrightarrow 0.
\tag{$B$}
\end{equation*}

\subsubsection*{Restriction to hyperplanes}
Let now $F$ be a divisor on $X^n_{(n-1)}\cong X^n_{(n-2)}$ and let $I(n-1)$ be the index set parametrising a hyperplane of $\PP^n$ spanned by $n$ points. 
Let $\tilde{H}_{I(n-1)}$ be the strict transform in $X^n_{(n-2)}$ of such a hyperplane, 
cfr. \eqref{exc of hyperplane} and \eqref{hyperplane tilde}.
We will call   \emph{sequence of type (C)} the following restriction sequence:
\begin{equation*}\label{seq C}
 0\longrightarrow F-\tilde{H}_{I(n-1)}\longrightarrow  F \longrightarrow
 {F}{|_{\tilde{H}_{I(n-1)}}}\longrightarrow 0.
\tag{$C$}
\end{equation*}

For each of the above sequences, it will be possible to study the restricted divisor by means of the intersection table \eqref{intersection table}.

\section{Standard Cremona transformations of hyperplane classes}\label{Cremona of hyperplane class}

We  recall that the \emph{standard Cremona transformation} of $\PP^n$, based at the $n+1$ coordinate points, is the birational transformation defined by the 
following rational map:
$$
\Cr: (x_0 : \dots : x_n) \to  (x_0^{-1}:\cdots:x_n^{-1}),
$$
where $x_0,\dots,x_n$ are homogeneous coordinates of $\PP^n$.
This map induces an action on the Picard group of $\PP^n$  blown-up at $s$ points, 
$X^n_{(0)}$. Without loss of generality we may assume that an effective divisor $D$ of the form \eqref{general divisor} is based on the $n+1$ coordinate points and other general points of the projective space and we label their corresponding exceptional divisors by
$E_1, \ldots, E_{n+1}, E_{n+2},\dots,E_s$. The Cremona action on the divisor 
$D$ is described by the following rule (see e.g. \cite{Dolgachev}). Set
$$D=dH-\sum_{i=1}^{s}m_i E_{i}, \quad c:=(n-1)d-\sum_{i=1}^{n+1}m_i.$$
Then 
$$
\Cr(D)=(d+c)H-\sum_{i=1}^{n+1}(m_i+c)E_i-\sum_{i=n+2}^{s}m_iE_i.
$$ 
In the case $n=3$,  $\Cr$ is often called the
 \emph{cubo-cubic} Cremona transformation, see for instance
\cite{laface-ugaglia-TAMS}.

\medskip

From now on throughout this section, $\mathcal{I}$ will denote the set of all subsets of  $\{1,\ldots, n+1\}$. Notice that 
$\mathcal{I}$ satisfies conditions (I), (II) and (III) of Section \ref{blow-up construction}, cfr. Remark \ref{permutohedron}.
Moreover $X^n_{(n-1)}\cong X^n_{(n-2)}$ will be the blow-up of $\PP^n$ along all linear subspaces parametrised by elements of $\II(n-2)$, i.e. along all coordinate subspaces of $\PP^n$, in increasing dimension.

The divisors \eqref{cremona}, that naturally arise in the blowing-up  construction, are the strict transforms in $X^n_{(n-1)}$  via the standard Cremona transformations of the strict transform $H$ of the hyperplane class $\mathcal{O}_{\PP^n}(1)$, where we abbreviate the notation for the strict transform in $X^n_{(n-1)}$ of $\Cr(H)$ by 
\begin{equation}\label{cremona divisor}
\Cr_n(H)=nH-\sum_{I(\rho)\in\II(n-2)}(n-\rho-1)E_{I(\rho)}.
\end{equation}
This is the divisor \eqref{cremona}  obtained by replacing 
$\rho$ by $n$, $t$ by $\rho$, $I$ by $\{1,\dots, n+1\}$ and $h_b$ by $H$.

In this section we compute all cohomologies of any multiple of Cremona transformations of hyperplane classes. In
particular we show that they have the same cohomological behaviour as the same 
multiples of the hyperplane classes. 

\medskip 

Recall that the canonical divisor of the blown-up projective space $X_{(n-1)}^n$ is
\[K_{X^{n}_{(n-1)}}=-(n+1)H+\sum_{I(\rho)\in\II(n-2)}(n-\rho-1)E_{I(\rho)}\]\
and notice that 
$$K_{X^{n}_{(n-1)}}+ \Cr_{n}(H)=-H.$$

\begin{theorem}\label{vanishing cremona} 
For any integer $a$, we have that $h^i(X^n_{(n-1)},\mathcal{O}_{X^n_{(n-1)}}(a\Cr_n(H)))=h^i(\PP^n, \mathcal{O}_{\PP^n}(a))$.

 In particular, we have that $h^i(X^n_{(n-1)},\mathcal{O}_{X^n_{(n-1)}}(-\alpha\Cr_n(H)))=0$, for $1\le \alpha\le n$, $i\ge 0$.
\end{theorem}
\begin{proof}
If $a=0$ the statement is obvious. 

Assume  that $a\ge1$.
It is known that (see for example \cite[Theorem 3]{Dumnicki})
$h^0(a\Cr_n(H))=h^0(aH)$.
Moreover, the effective divisor $a\Cr_n(H)$ on $X^n_{(n-1)}$
 is not obstructed, that is $h^i(a\Cr_n(H))=0$, for all $i\ge1$, see 
\cite[Theorem 4.6]{BraDumPos}.

Assume now that $a\le -1$ and denote $a=-\alpha$, where $\alpha$ is a positive integer.

Case $\alpha\geq n+1.$
Recall that $h^i(\PP^n, \mathcal{O}(-\alpha))=0$ for $i\neq n$ and that, by Serre duality, $h^n(\PP^n, \mathcal{O}(-\alpha))=h^{0}(\PP^n,\mathcal{O}(\alpha-n-1))>0$.
Notice that, on the blown-up space $X^n_{(n-1)}$, we have the following:
\begin{align*}
h^i(-\alpha \Cr_n(H)))&=h^{n-i}((\alpha-1)\Cr_{n}(H)-H)\\
&=h^{n-i}\left([(\alpha-1)n-1]H-(\alpha-1)\sum_{I(\rho)\in\II(n-2)}(n-\rho-1)E_{I(\rho)}\right),
\end{align*}
where the first equality follows from Serre duality and the second equality is just the expanded form of  $(\alpha-1)\Cr_{n}(H)-H$.
We claim that the cohomologies vanish for all $i\ne n$. To show this, we
notice that the divisor 
$$((\alpha-1)n-1)H-(\alpha-1)(n-1)\sum_{i=1}^{n+1}E_{i}$$ on $X^n_{(0)}$ is effective and that each cycle $L_{I(\rho)}$ is contained in its base locus with multiplicity 
$k_{I(\rho)}=(\alpha-1)(n-\rho-1)+\rho$,  by \cite[Lemma 2.1]{BraDumPos}. Each integer $k_{I(\rho)}$ differs from the coefficient of $E_{I(\rho)}$ in the above expression by $\rho$. Hence the vanishing of the $i-$th cohomology group, for all $i\neq n$, follows by  \cite[Theorem 4.6]{BraDumPos}.
If $i=n$, notice that
$$h^n\left(X^n_{(n-1)},-\alpha \Cr_n(H)\right)=h^{0}\left(X^n_{(0)}, ((\alpha-1)n-1)H-(\alpha-1)\sum_{j=0}^{n+1}(n-1)E_{j}\right).$$
We compute the number of global sections by performing a standard Cremona transformation in $\PP^n$, which preserves that number:
\begin{align*}
h^n\left(X^n_{(n-1)},a \Cr_n(H)\right)&=h^{0}\left(X^n_{(0)},(\alpha-n-1))H+(n-1)\sum_{j=0}^{n+1}E_{j}\right)\\
&=h^{0}\left(X^n_{(0)},(\alpha-n-1)H\right).
\end{align*}

Case $1\le \alpha\le n$.  We prove the statement by induction on $n$ and $\alpha$.

The base steps $n=2$, $\alpha=1,2$ are easily verified by means of Serre duality. Indeed, 
as the canonical divisor of $X^2_{(0)}$ is $K=-3H+E_1+E_2+E_3$, we have $h^i(-\Cr_2(H))=h^{2-i}(-H)=0$ and $h^i(-2\Cr_2(H))=h^{2-i}(H-E_1-E_2-E_3)=0$.

Fix $n\ge 3$ and recall from \eqref{IIj} that $\mathcal{I}(n-2)_1$ denotes the set of all subsets of $\{1,\dots,n\}$ of cardinality at most $n-1$ containing $\{1\}$ as a proper subset.
Using \eqref{exc of hyperplane}, we compute the following equality:
\begin{align*}
-\sum_{1\in I(n-1)}E_{I(n-1)}=&-nH+ 
\sum_{I(\rho)\in\II(n-2)_1}(n-\rho)E_{I(\rho)}\\
&+
\sum_{I(\rho)\in\II(n-2)\setminus\II(n-2)_1}(n-\rho-1)E_{I(\rho)}
\\= & -\Cr_n(H)+E_1+\sum_{I(\rho)\in\II(n-2)_1}E_{I(\rho)}.
\end{align*}
For $1\le r\le n-1$, and $I=I(r)$ we recall from \eqref{intersection table} that 
$E_{I(\rho)}|_{E_{I}}=e_{I(\rho)}\boxtimes0$ if $I(\rho)\subset I$ and that
$E_{I(\rho)}|_{E_{I}}=0\boxtimes\ast$ if $I(\rho)\nsubseteq I$, where we use $\ast$ to denote
 the appropriate divisor, as we are only interested in the first factor.
Hence, the above computation and Remark \ref{exc of hyperplanes} show that
$$-\sum_{1\in I(r-1)}E_{I(r-1)}|_{E_{I}}=\left(
-\Cr_r(h)+e_1+\sum_{I(\rho)\in\II(r-2)_1}e_{I(\rho)}\right)\boxtimes\ast.
$$
Therefore, 
\begin{equation}\label{cremona on restriction}
\left(-E_1-\sum_{I(\rho)\in\II(r-1)_1}{E_{I(\rho)}}\right){|_{E_I}}=-\Cr_r(h)\boxtimes*.
\end{equation}

Now, for every integer $r$ such that $0\le r\le n-1$, we will consider the ordered set  of all multi-indices 
of length $r+1$ that contain $\{1\}$. We will denote it by $\{I(r)_0,\dots, I(r)_{s_r}\}\subset\mathcal{I}(n-1)_1$, where $s_r+1$ is its cardinality.
Notice that for $r=0$ we have $s_0=0$ and the only set that we consider is the singleton $\{1\}$. 
Let $\prec$ be the lexicographical order on $\II(n-1)_1$ defined as follows: 
$I(r')_{j'}\prec I(r)_{j}$ if and only if $r'<r$ or $r'=r$ and $j'<j$.

In the space $X^n_{(n-1)}$ we consider the divisors $F(r,j)$ defined by recursion starting from $F:=0$ as follows:

\begin{align*}
F(0,0)=&\ F-E_1,\\
F(r,0)=&\ F(r-1,s_{r-1})-E_{I(r)_0}, \ 1\le r\le n-1,\\
F(r,j)=&\ F(r,j-1)-E_{I(r)_j}, \ 1\le j\le s_r.
\end{align*}
Notice that the last divisor is $F(n-1,s_{n-1})=-\Cr_n(H)$.

The divisor $F(0,0)$ is obtained from the divisor $F$ as the kernel of an exact sequence of type \eqref{seq A}. 
For $1\le r\le n-2$ each divisor $F(r,j)$ is obtained from the previous 
one as kernel of a sequence of type \eqref{seq B} and for $r=n-1$ from a sequence of type 
\eqref{seq C}, cfr. Section \ref{sequences}.

Precisely, we consider the following exact sequences of sheaves, performed following the order $\prec$ on $\II(n-1)_1$:
\begin{align*}
0\to F(0,0)\to F\to{F}{|_{E_1}}\to0,& \\
0\to F(r,0)\to F(r-1,s_{r-1})\to{F(r-1,s_{r-1})}{|_{E_{I(r)_1}}}\to0,& \quad 1\le r\le n-1\\
0\to F(r,j)\to F(r,j-1)\to{F(r,j-1)}{|_{E_{I(r)_{j}}}}\to0,& \quad 1\le j\le s_r.
\end{align*}

The divisor $-E_1$ has vanishing cohomologies, for all $n\ge1$.
Moreover, in all sequences the restricted divisor is of the form 
\eqref{cremona on restriction} and has therefore vanishing cohomologies, 
by induction on $n$, using the  Kunneth formula for the cohomology of factors. This implies that 
$H^i(-\Cr_n(H))=0$, $i\ge 0$, concluding the proof of the statement in the case $\alpha=1$.

We are left to prove the vanishing theorems for $-\alpha\Cr_{n}(H)$, $2\le \alpha\le n$.
For $\alpha\ge 2$, we assume the statement true for $\alpha-1$. We apply the recursive restriction procedure as above. Setting $F:=-(\alpha-1)\Cr_n(H)$, we get $F(n-1,s_{n-1})=-\alpha\Cr_n(H)$.

We first consider the first sequence, the restriction to $E_1$. Notice that the space $E_1\cong X^{n-1}_{(n-2)}$ is the $(n-1)-$dimensional toric variety  described in Remark \ref{permutohedron}.
Now, we first notice  that  $-\Cr_n(H){|_{E_1}}=-\Cr_{n-1}(h)$ on $E_1\cong X^{n-1}_{(n-2)}$.
 Indeed,
$$
-nH+(n-\rho-1)\sum_{I(\rho)\in\II(n-2)}{E_{I(\rho)}}{|_{E_1}}=
-(n-1)h+(n-\rho-1)\sum_{I(\rho)\in\II(n-2)\setminus\II(1)}{e_{I(\rho)|1}}.
$$
Therefore, the restricted divisor of the first sequence, that is  $-(\alpha-1)\Cr_n(H){|_{E_1}}=-(\alpha-1)\Cr_{n-1}(h)$ on $E_1$,
has vanishing cohomologies by induction on $n$, for all $\alpha\le n-1$.

Moreover, a computation similar to that preceding  \eqref{cremona on restriction} 
shows that, for all $I=I(r)_j$,  $-{\Cr_{n}(H)}{|_{E_I}}=0\boxtimes*$.
Therefore, for every pair $(r,j)$ the restriction of the corresponding sequence is
$$
\left(-(\alpha-1)\Cr_n(H)-E_1-\sum_{\substack{I(\rho)\prec I \\ I(\rho)\in\II(n-1)_1}}E_{I(\rho)}\right){|_{E_I}},
$$
that equals the divisor on the right hand side of \eqref{cremona on restriction}. 
It has vanishing cohomologies by the argument above (case $\alpha=1$).
This concludes the proof.
\end{proof}

\section{Vanishing theorems for effective only linearly obstructed  divisors}\label{vanishing section n+3}

Let $\LL=\LL_{n,d}(m_1,\dots,m_s)$  be a linear system of hypersurfaces of degree $d$ of $\PP^n$ interpolating $s$ points with assigned  multiplicities $m_i$'s.
 Elements of $\LL$ are in bijection with divisors 
on $X_{(0)}^n$  of the  form \eqref{general divisor},
$$
 D:= dH-\sum_{i=1}^s m_i E_i.
$$

Assume that the following bound is satisfied: 
\begin{equation}\label{EffectivityCondition}
\sum_{i=1}^s m_i-nd\le \min(n-s(d),s-n-2), \quad 1\le m_i\le d,
\end{equation}
where the integer $s(d)$ denotes the number of indices $i$ such that $m_i=d$.
 
When a given list of integers $(d,m_1,\dots,m_s)$ satisfies condition \eqref{EffectivityCondition}, if $D$ is the divisor defined by these integers as in \eqref{general divisor}, we will say that $D$ \emph{satisfies condition} \eqref{EffectivityCondition}.

Condition \eqref{EffectivityCondition} is sufficient condition for $D$  to be effective 
(and for $\ls$ to be non-empty), see
Theorem \ref{theorem a n+3}.

In \cite[Section 5]{BraDumPos}  the dimensions of all linear systems $\ls$ in $\PP^n$  (equivalently
the number of basis elements of global sections of the line bundles associated to the divisors $D$) 
that satisfy the bound \eqref{EffectivityCondition} are given, see also Theorem \ref{theorem a n+3}.
In this section we compute the dimension of all higher cohomology groups of their strict transforms in the blown-up spaces $X_{(r)}$, using the notation introduced in Section \ref{blow-up construction}.

\begin{theorem}\label{vanishing for >=n+3}
Statements $(a)$ and $(b)$ of Theorem \ref{monster for effective} hold for divisors satisfying condition 
\eqref{EffectivityCondition}.
\end{theorem}

This result implies  that such linear systems are only linearly obstructed.
Moreover this (partially) answers Question \ref{question}
for divisors satisfying \eqref{EffectivityCondition}.

\medskip

This section is organised as follows. 
In Section \ref{linear base locus lemma} we prove a base locus lemma, Proposition \ref{sharpBLL},
 that computes the exact multiplicity of containment 
of a linear cycle in the base locus of a linear system $\ls$. Section \ref{Vanishing theorems section} contains the proof of Theorem \ref{vanishing for >=n+3}.

\subsection{Linear base locus lemma}\label{linear base locus lemma}

Let $\II$ be the set of all subsets of $\{1,\dots,s\}$. 

For all $I(r)=\{i_1,\dots,i_{r+1}\}\in \II$ with $0\le r\le \min(n,s)-1$, we introduce the integers
 (cfr. Definition \eqref{mult k}):
\begin{align*}
K_{I(r)}=K_{i_1,\dots,i_{r+1}}&:=\sum_{i\in I(r)}m_i -rd;\\
k_{I(r)}=k_{i_1,\dots,i_{r+1}}&:=\max (K_{I(r)},0 ).
\end{align*}
Moreover we introduce the integer  
$$ \bar{r}=\bar{r}(\LL)=\bar{r}(D):=\max(\rho|K_{I(\rho)}> 0),$$ 
and  the following subset of $\II$:
\begin{equation}\label{II}
\II^{>}=\II^{>}(\ls)=\II^{>}(D):=\{I(r)\subseteq\{1,\dots,s\}:0\le r\le n-1, K_{I(r)}>0\}.
\end{equation}

\begin{proposition}\label{sharpBLL}
Let $\LL=\LL_{n,d}(m_{1},...,m_{s})$ be a non-empty linear system. Let 
$I(r)\in\II^>$. The linear cycle $L_{I(r)}$ spanned by the points 
parametrised by $I(r)$ is contained with multiplicity $k_{I(r)}$
in the base locus of $\LL$. 
\end{proposition}

We can rephrase Proposition \ref{sharpBLL} in the setting of the blow-up $X^n_{(0)}$ by 
saying that the strict transform of  $L_{I(r)}$  is contained with multiplicity $k_{I(r)}$
in the base locus of $D$.

\begin{proof}
Let $1\le r \le n-1$ and let $I(r)$ be  a multi-index with  $k_{I(r)}=K_{I(r)}\ge 0$
and let $\tilde{k}_{I(r)}\ge0$ be the multiplicity with which $L_{I(r)}$ is contained in the base locus
of $\LL$.
By B\'ezout's theorem one has $\tilde{k}_{I(r)}\ge k_{I(r)}$, 
see \cite[Lemma 2.1]{BraDumPos} for details.
We introduce the following notation:
$$R=R(\LL,I(r)):=\max(\rho\leq n-1|K_{I(\rho)}\geq 0, I(r)\subseteq I(\rho)).$$

If $r=n-1(=R)$ then the claim is true by {\cite[Lemma 4.4]{CDD}}.
For $r\le n-2$ we consider separately the following cases:
\begin{itemize}
\item[(i)] $r= R=\bar{r}$,
\item[(ii)] $r= R<\bar{r}$,
\item[(iii)] $r<R$.
\end{itemize}

\medskip

Case (i).
We prove the statement 
by  backward induction on $R$.
Precisely, given $R\le n-2$, we assume that for every non-empty 
linear system $\mathcal{M}$ in $\PP^n$ such that
$R(\mathcal{M},I(r))=R+1$ and for
every multi-index $I(R+1)$ with $K_{I(R+1)}\ge0$,
 the cycle $L_{I(R+1)}$ is contained in the base locus of $\mathcal{M}$
 with multiplicity $K_{I(R+1)}$, and we
prove the statement for $\LL$ and $I(r)$ with $R(\LL,I(r))=R$.
Let $I(R)=\{i_1,\ldots, i_{R+1}\}$ be a multi-index with
 $K_{I(R)}\ge 0$, and consider the inclusions
 $I(R)\subset J=\{i_1,\dots, i_{R+1},i_{R+2}\}$ and 
$J\subset \bar{J}=\{i_1,\dots, i_{R+1},i_{R+2}, \dots, i_{\min(s,n)}\}$.  
Because $K_{J}\leq 0$, then by induction the cycle $L_{J}$ is not  in the base locus of $\mathcal{L}$.
Let $H_{\bar{J}}:=H-\sum_{i\in \bar{J}}E_i$ denote the strict transform 
on $X^n_{(0)}$ of the hyperplane of $\PP^n$
spanned by the points indexed by elements of $\bar{J}$.
We introduce the following divisor:
$$D':=D+(-K_{J})H_{\bar{J}}.$$
Notice that 
$D'$ is an effective divisor on $X^n_{(0)}$, of the form \eqref{general divisor}. We will denote by $m'_i$ 
and $d'$ its coefficients: $d':=d-K_{J}$, $m'_{i}:=m_{i}-K_{J}$ 
for $i\in \bar{J}$ and $m'_{i}:=m_{i}$ for $i\notin \bar{J}$. By construction,
$$K'_{J}:=\sum_{i\in J} m'_{i}-(R+1)d'=0.$$
Moreover, $K'_{I(R)}=K_{I(R)}+(-K_J)\ge K_{I(R)} \ge0$, so that 
 $L_{I(R)}$ is contained in the base locus of $D'$ with multiplicity at least $K'_{I(R)}$.

Assume that $\tilde{k}_{I(R)} \geq K_{I(R)}+1$. The multiplicity of containment of the linear cycle $L_{I(R)}$ 
in the base locus of $D'$ is $$\tilde{k}'_{I(R)}\geq \tilde {k}_{I(R)}+(-K_{J})\geq K_{I(R)}+1+(-K_J)= K'_{I(R)}+1.$$
For any point $p\in L_{I(R)}$, we compute the multiplicity $\tilde{k}'_{p, p_{R+2}}$ with which the
 line spanned by the points $p$ and $p_{i_{R+2}}$ is contained in the base locus of $D'$:
$$
\tilde{k}'_{p, p_{R+2}}\geq \tilde{k}'_{I(R)}+m'_{R+2}-d'\ge K'_{I(R)}+1+m'_{R+2}-d'=K'_{I(R+1)}+1=1.
$$
This shows that such a  line  is in the base locus of $D'$ for any $p\in L_{I(R)}$. Letting $p$ vary in $L_{I(R)}$, we obtain that $L_J$ is in the base locus of $D'$. This gives a contradiction.

\medskip

Case (ii).
We know by the previous case that for all cycles $L_{I(\bar{r})}$, the multiplicity of containment is given by $k_{I(\bar{r})}.$ For smaller cycles, $I(R)$, with
$1\leq r=R<\bar{r}$ and $K_{I(R)}\geq 0$ we run induction on $\bar{r}- R$ and the same argument used for Case (ii) applies.
 
\medskip

Case (iii). We assume that the statement holds for $I(R)$, with $I(r)\subseteq I(R)$ and  such that $K_{I(R)}\ge0$ and we show that it holds for
$I(r)$. Namely, assuming 
$\tilde{k}_{I(R)}=k_{I(R)}$, for all $I(R)$, we prove that $\tilde{k}_{I(r)}=k_{I(r)}$.
Notice that $0\le K_{I(R)}\leq K_{I(r)}$, since $m_i\le d$.
Therefore, all linear subspaces $L_{I(\rho)}$ of $L_{I(R)}$
are contained in $\Bs(\LL)$ with multiplicity at least 
$K_{I(\rho)}\geq 0$,  by {\cite[Lemma 2.1]{BraDumPos}}. In particular,
$L_{I(r)}$ is contained at least $K_{I(r)}$ times.

Assume now by contradiction that $L_{I(r)}$ is contained in $\Bs(\LL)$ with 
multiplicity at least $1+K_{I(r)}$. 
We know that the linear cycle $L_{I(R)\setminus I(r)}$ is contained in the base locus 
with multiplicity at least 
$K_{I(R)\setminus I(r)}\ge 0$. 
For any point $p$ in the cycle $L_{I(r)}$ and $p'$ in $L_{I(R)\setminus I(r)}$ 
we obtain that the line spanned by $p$ and $p'$ is contained in the base locus 
with multiplicity at least $1+K_{I(r)}+K_{I(R)\setminus I(r)}-d=1+K_{I(R)}$. 
This gives a contradiction.
\end{proof}

\begin{remark}
We must mention that the first part
of the proof of Proposition \ref{sharpBLL}, Case (i), 
was established in {\cite[Proposition 2.5]{BraDumPos}}, but we include it here also for 
the sake of completeness.
\end{remark}

\subsection{Vanishing theorems}\label{Vanishing theorems section}
In this section we prove Theorem \ref{vanishing for >=n+3}.
The proof will be based on induction on $n\ge1$ and 
on the multiplicities. The case $n=1$ is trivial. The case $s\le n+2$, namely 
$m_i=0$, for $i\ge n+3$,
is solved in \cite[Ch. 4]{BraDumPos}.

As in Section \ref{linear base locus lemma}, let $\II$ be the set of all subsets of $\{1,\dots,s\}$ and  let $\II^>$ be the subset of $\II$ introduced in \eqref{II}.
For all $1\le r\le n-1$, let $\II(r)\subset\II$ be the set of multi-indices of $\II$ of length at most $r+1$, as in \eqref{IIj}.

 \begin{lemma}\label{I II III}
 In the above notation, if $D$ satisfies \eqref{EffectivityCondition}, then the set $\II^>(D)$ satisfies conditions \emph{(I), (II)} and \emph{(III)} of Section \ref{blow-up construction}.
 \end{lemma}
\begin{proof}
Condition (I) follows by the definition. 

Since $m_i\le d$, for all $i$, then $K_I\ge K_J$ for $I\subset J$. Hence (II) is satisfied. 

If $|I\cup J|\le n+1$ then (III) follows easily, because one may assume that $L_I$ and $L_J$ are coordinate subspaces. 
Assume that $I\cap J=\emptyset$ and $|I\cup J|= n+2$. 
The inequalities \eqref{EffectivityCondition} and $m_i\ge1$ for all $i$, 
imply that  $K_I+K_J=\sum_{i\in I\cup J}m_i-nd\le0$, hence at most one among $I$ and $J$ is in $\II^>$. 
This proves (III) in this case. If $I\cap J\neq\emptyset$ and $|I\cup J|= n+2$, we conclude by just noticing that $K_I+K_J\le K_I+K_{J\setminus I}\le0$. We leave it to the reader to verify that condition (III) holds also for $|I\cup J|\ge n+3$.
\end{proof}

In the notation of Section \ref{blow-up construction}, for every $1\le r\le n-1$, 
let $X^n_{(r)}$ be the blow-up of $X^n_{(r-1)}$  along 
the union of the strict transforms of the $r$-cycles $L_{I(r)}$, $I(r)\in\mathcal{I}^>$. The total transform of $D_{(r-1)}\subset X^n_{(r-1)}$ is 
\begin{equation}\label{complete transform}
(\pi_{(r)})^\ast D_{(r-1)}=dH-\sum_{I(\rho)\in\II(r-1)} k_{I(\rho)} E_{I(\rho)},
\end{equation}
while the strict transform of $D_{(r-1)}$ (cfr. \eqref{proper transform 1}) is
\begin{equation}\label{strict transform}
\begin{split}
D_{(r)}&= dH-\sum_{I(\rho)\in\II(r-1)} k_{I(\rho)} E_{I(\rho)}-\sum_{I(r)\in\II} k_{I(r)} E_{I(r)}\\
& =
dH-  \sum_{I(\rho)\in\II(r)} k_{I(\rho)} E_{I(\rho)}.
\end{split}
\end{equation}

\subsubsection{Induction on the sum of the multiplicities}\label{Induction on b}

Let $D$ be as in \eqref{general divisor}. Modulo reordering the indices $\{1,\dots,s\}$ if necessary, we may assume that $m_1\ge1$.
We introduce the following divisors on $X^n_{(0)}$:
$$
D':=D+E_1=dH-(m_1-1)E_1-\sum_{i=2}^sm_i E_i.
$$
It corresponds to the linear system of hypersurfaces of $\PP^n$ denoted by $\ls':=\ls_{n,d}(m_1-1,m_2,\dots,m_s)$.

\begin{remark}\label{D' has smaller base locus}
Notice that $\II^>(D')\subset\II^>(D)$. In particular the linear base locus of $D'$ is contained in that of $D$. Precisely,
the cycles $L_I$ with $I\notin \mathcal{I}(n-1)_1$ have the same 
multiplicity of containment in both base loci, 
while the cycles $L_I$ with $I\in\mathcal{I}(n-1)_1$ are contained with multiplicity one more in $D$, 
see Proposition \ref{sharpBLL}. 
\end{remark}

The following provides the induction step on the integer $b=\sum_{i=1}^sm_i-nd$, that was defined in \eqref{b}.

\begin{lemma}\label{condition b kernel}
If $D$ satisfies \eqref{EffectivityCondition}, then so does $D'$. In particular the set $\II^>(D')$ satisfies condition  \emph{(I), (II)} and \emph{(III)} of Section \ref{blow-up construction}.

\end{lemma}
\begin{proof}
It is a trivial computation that $b(D')=b(D)-1$. The second statement follows from Lemma \ref{I II III}.
\end{proof}

\medskip

For all $I\in \mathcal{I}(r)_1$ (cfr. definition \eqref{IIj}), set $k'_I:= k_I(D')=\max(K_I-1,0)$. 
We have
\begin{equation}\label{D'r}
D'_{(r)}=dH-(m_1-1)E_1-\sum_{i=2}^sm_i E_i-
\sum_{I\in\mathcal{I}(r)_1}k'_I E_I- 
\sum_{I\in\mathcal{I}(r)\setminus\mathcal{I}(r)_1}k_IE_I.
\end{equation}
Using the notation 
$$\mathcal{I}(r)_1^>:=\II^>\cap\II(r)_1=\{I(\rho) : 1\le \rho \le r, 1\in I(\rho), K_{I(\rho)}>0\},$$
 we can write
$$
D'_{(r)}= D_{(r)}+\sum_{I\in \mathcal{I}(r)_1^>}E_I.
$$

\begin{remark}\label{D' lives in spaces of D}
For every $1\le r\le n-1$, the natural space where $D'_{(r)}$ lives is the
 subsequent blow-up of $\PP^n$ along the linear cycles parametrised by $\II^>(D')$, ordered in increasing dimension. We can consider the total transform of $D'_{(r)}$ in the subsequent blow-up of the latter space along the linear cycles parametrised by $\II^>(D)\setminus \II^>(D')$, in increasing dimension: this space is denoted by $X^n_{(r)}$.
The dimensions of the cohomology groups of $D'_{(r)}$ and of its total  transform are equal. Therefore, abusing notation, in this section and throughout this paper, we will use the same symbol for $D'_{(r)}$ and for its total transform in $X^n_{(r)}$ and we will consider the cohomology groups
$H^i(X^n_{(r)},\mathcal{O}_{X^n_{(r)}}(D'_{(r)}))$.
\end{remark}

\subsubsection{Induction on $n$}\label{Induction on n}
Let us assume, without loss of generality, that  $d\ge m_1\ge m_2\ge\cdots\ge m_s\ge1$.
In order to employ induction on $n$ we want to restrict the divisor $D'_{(r)}$ to
the strict transform of the exceptional divisor $E_1$, 
that is isomorphic to the blown-up space $X^{n-1}_{(r-1)}$. Such a restricted divisor will  also satisfy  condition \eqref{EffectivityCondition}.

Precisely, for every $1\le r\le n-1$, in the space $X^n_{(r)}$  we will use the following restriction sequence of type  \eqref{seq A}:
\begin{equation}\label{seq Ar}
 0\longrightarrow D'_{(r)}-E_1\longrightarrow  D'_{(r)}\longrightarrow
 {D'_{(r)}}{|_{E_1}}\longrightarrow 0.
\end{equation}
of which now we give a detailed description. 
By abuse of notation, write $E_1$ for the strict transform $(\pi_{(r,0)})^{-1}_\ast E_1$ 
of the exceptional divisor in $X^n_{(0)}$ of the point $p_1$. Recall from Section 
\ref{exc point} that 
$E_1\cong X^{n-1}_{(r-1)}$ has Picard group generated by the hyperplane class $h$ and by the exceptional classes
$e_{I(\rho)|1}={E_{I(\rho)}}{|_{E_1}}$, for $I(\rho)\in\mathcal{I}(r)_1$. 
The divisor $D'_{(r)}$ restricts to $E_1$ as
$$
{D'_{(r)}}{|_{E_1}}=(m_1-1)h-
 \sum_{I\in \mathcal{I}(r)_1}k'_{I}e_{{I|1}},
$$
where $k'_{I}:=\max(K_I-1,0)$, as in Section \ref{Induction on b}.

\begin{remark}
Notice that for $r=1$, $E_1\cong X^{n-1}_{(0)}$ is the space $\PP^{n-1}$ 
blown-up at a collection of points in general position.
Therefore $G:={D'_{(r)}}{|_{E_1}}$ can be seen as a divisor of the form \eqref{general divisor} 
given by the integers $(m_1-1,k'_{12},\dots,k'_{1s})$. Moreover one can check that
$G_{(r-1)}=D'_{(r)}{|_{E_1}}$, for every $1\le r\le n-1$.
\end{remark}

The following provides the induction step on $n$; in fact ${D'_{(1)}}{|_{E_1}}$ is a divisor 
of the form \eqref{general divisor}  in a blown-up $\PP^{n-1}$ in points in general position.

\begin{lemma}\label{cond b n+3}
If $D$ satisfies \eqref{EffectivityCondition}, then so does ${D'_{(1)}}{|_{E_1}}$.
\end{lemma}
Even though the same argument appeared in the proof of \cite[Lemma 5.7]{BraDumPos},
we include it here for the sake of completeness.
\begin{proof}

Set $\bar{s}:=\#\II(1)_1^>$ be the number of lines contained in the base locus of $D$ passing through the first point $p_1$.
Consider the divisor  $${D'_{(1)}}{|_{E_1}}=(m_1-1)h-
 \sum_{i=2}^{\bar{s}}k'_{1i}e_{{1i|1}}$$ in $E_1\cong X^{n-1}_{(0)}$.
Notice that $k'_{1i}=m_1+m_i-d-1\le m_1-1$, 
as $m_i\le d$, for all $i=2,\dots,\bar{s}$.
If $\bar{s}\le n-1$, the first inequality of \eqref{EffectivityCondition} is trivially satisfied by ${D'_{(1)}}{|_{E_1}}$.
When  $\bar{s}\ge n$, we conclude by computing
\begin{align*}
\sum_{i=2}^{\bar{s}}k'_{1i}&
=\sum_{i=2}^{\bar{s}} m_i+\bar{s}(m_1-d-1)\\
&\le\sum_{i=1}^sm_i-m_1-(s-\bar{s}-1)+\bar{s}(m_1-d-1)\\
&\le nd+s-n-m_1-(s-1)+\bar{s}(m_1-d)\\
&\le(n-1)(m_1-1).
\end{align*}
\end{proof}

\subsubsection{Global sections in the exact sequences  \eqref{seq Ar}}

Let us denote by
\begin{equation}\label{linear obstruction}
l(D,r):=\sum_{I(\rho)\in\II\setminus \II(r-1)}
(-1)^{\rho-r-1}{{n+k_{I(\rho)}-\rho-1}\choose n},
\end{equation}
 the integer that appears in Theorem \ref{monster for effective} $(b)$.
It is the alternating sum of
the contributions to the formula for the linear virtual dimension of $\ls=\ls_{n,d}(m_1,\dots,m_s)$, given by the multiple base cycles of dimension at least $r+1$.
Notice that $l(D,0)=\ldim(\ls)$, see Definition \ref{new-definition}.

\begin{lemma}\label{sum of obstructions}
In the above notation, the following equality holds: 
$$l(D,r)=l(D',r)-l({D'_{(1)}}{|_{E_1}},r-1).$$
\end{lemma}
\begin{proof}
It follows from the equality of Newton binomials ${\alpha\choose\beta}={{\alpha-1}\choose \beta}+{{\alpha-1}\choose{\beta-1}}$, for $\alpha,\beta\ge1$.
\end{proof}

\begin{proposition}\label{h^0 exact}
Assume that a divisor $D$ of the form \eqref{general divisor} satisfies condition \eqref{EffectivityCondition} and fix  $1\le r\le n-1$. 
The following sequence of global sections is exact:
$$
0\longrightarrow H^0(D'_{(r)}-E_1)\longrightarrow  H^0(D'_{(r)})\longrightarrow
 H^0({D'_{(r)}}{|_{E_1}})\longrightarrow 0.
$$

\end{proposition}
\begin{proof}
It is enough to show the statement for $r=1$, because the number of global sections of $D_{(1)}$ and $D'_{(1)}$
(and of ${D_{(1)}}{|_{E_1}}$) are preserved when
taking strict transforms in the spaces $X^n_{(r)}$ ($E_1\cong X^{n-1}_{(r-1)}$ respectively). 

Notice that  $h^0(D'_{(1)}-E_1)=h^0(D'-E_1)$. Since $D'-E_1=D$ by definition, we conclude that $h^0(D'_{(1)}-E_1)=h^0(D)$.

Moreover, notice that ${D_{(1)}}{|_{E_1}}$ is a divisor on $E_1\cong X^{n-1}_{(0)}$, the $(n-1)$-dimensional projective space blown-up at points.

One can verify that  $\ldim(D)=\ldim(D')-\ldim({D'_{(1)}}{|_{E_1}})$, using Lemma \ref{sum of obstructions}.
Furthermore, by Theorem \ref{theorem a n+3}  all three divisors are effective 
and only linearly obstructed, as they satisfy 
\eqref{EffectivityCondition} by Lemma \ref{condition b kernel} and Lemma \ref{cond b n+3}. 
Hence $h^0(D)=h^0(D')-h^0({D'_{(1)}}{|_{E_1}})$. 
\end{proof}

\subsubsection{The cohomologies of the kernel of the sequences \eqref{seq Ar}}

We now compute all cohomology groups of the divisors in the sequences
\eqref{seq Ar}, for every $r$, and in particular we obtain that all cohomology groups  but the $(r+1)-$st vanish.

\begin{proposition}\label{simple obstructions}
Assume that $D$ satisfies conditions \eqref{EffectivityCondition}  and fix  $1\le r\le n-1$. 
Assume that statement $(b)$ of  Theorem \ref{monster for effective} holds for 
the pairs $(D',r)$ and $({D'_{(1)}}{|_{E_1}},r-1)$.
Then statement $(b)$ of  Theorem \ref{monster for effective}  holds for  the pair $(D'_{(r)}-E_1,r)$.
\end{proposition}

\begin{proof}
For $r$ such that $1\le r\le \min(\bar{r},n-1)$, in the space $X^n_{(r)}$ we consider  sequence 
\eqref{seq Ar}. By assumption we have 
\begin{itemize}
\item $h^i({D'_{(r)}}{|_{E_1}})=0$, for all $i\ne 0,r$ and $h^r({D'_{(r)}}{|_{E_1}})=
l({D'_{(r)}}{|_{E_1}},r-1)$;
\item $h^{i+1}({D'_{(r)}})=0$, for all $i\ne 0,r$ and $h^{r+1}({D'_{(r)}})=l(D',r)$.
\end{itemize}
Therefore, the long exact sequence in cohomology associated with sequence \eqref{seq Ar} splits into
the fundamental sequences 
\begin{equation}\label{seq h0}
0\to H^0(D'_{(r)}-E_1)\to
H^0(D'_{(r)})\to H^0({D'_{(r)}}{|_{E_1}})\to H^1(D'_{(r)}-E_1)\to 0,
\end{equation}
and
\begin{equation}\label{seq hr}
0\to H^r({D'_{(r)}}{|_{E_1}})\to H^{r+1}(D'_{(r)}-E_1)\to
H^{r+1}(D'_{(r)})\to 0.
\end{equation}
Using Proposition \ref{h^0 exact}, we obtain  $h^1(D'_{(r)}-E_1)=0$ from  \eqref{seq h0}. Moreover 
\begin{equation}\label{from tildeD to Dr}
h^{r+1}(D'_{(r)}-E_1)=h^r({D'_{(r)}}{|_{E_1}})+h^{r+1}(D'_{(r)})=l(D,r).
\end{equation}
The first equality is obtained by means of the sequence \eqref{seq hr}, the second equality follows from Lemma \ref{sum of obstructions}.
In particular,  $h^i(D'_{(r)}-E_1)=0$, for all $i\ne 0,r+1$ and this concludes the proof.
\end{proof}

\subsubsection{The cohomology of $D_{(r)}$}

Now we  deduce, from the above results, the cohomologies of $D_{(r)}$ for all $r\le n-1$.
 
\begin{proposition}\label{sequence B}
Assume that $D$ is effective.
For all $0\le r\le n-1$ and $i\ge 0$ we have 
$$h^i(D'_{(r)}-E_1)=h^i(D_{(r)}).$$
\end{proposition}

\begin{proof}
If $r=0$ the statement is obvious. 

Fix $1 \le \rho\le r\le n-1$.  
For every $\rho$, let $\{I(\rho)_0,\dots, I(\rho)_{s_\rho}\}$ be the set of multi-indices 
of length $\rho+1$ that contain $\{1\}$.
We use the notation $\mathcal{I}(\rho)_1^>=\{I\in\mathcal{I}(\rho)_1 :K_I>0\}$ 
of Section \ref{Induction on b}.

Let $\prec$ be the total order on $\mathcal{I}(\rho)_1^>$ inherited by the lexicographical order on $\mathcal{I}(n-1)_1$,
that was introduced  in the proof of Theorem \ref{vanishing cremona}. Precisely $I(\rho')_{j'}\prec I(\rho)_j$ if 
and only if $\rho'<\rho$ or $\rho'=\rho$ and $j'<j$.

In the space $X^n_{(r)}$, for every pair $(\rho,j)$, $1\le\rho\le r$ and $0\le j\le s_\rho$, for which $I(\rho)_j\in \mathcal{I}(\rho)_1^>$, we consider the divisor
\begin{equation}\label{F}
F(\rho,j)=D'_{(r)}-E_1-
\sum_{\substack{I(\rho')_{j'}\prec I(\rho)_{j}\\
I(\rho')_{j'}\in \mathcal{I}(\rho)_1^>}}E_{I(\rho')_{j'}}-E_{I(\rho)_{j}},
\end{equation}
 that we recursively define, using the same idea as in the proof of Theorem
 \ref{vanishing cremona}, using sequences of type \eqref{seq A} and \eqref{seq B}, by  the following rule:
\begin{equation}\label{recursion}
\begin{split}
F(0,0)=&\ D'_{(r)}-E_1,\\
F(\rho,0)=&\ F(\rho-1,s_{\rho-1})-E_{I(\rho)_0}, \ 1\le \rho\le r,\\
F(\rho,j)=&\ F(\rho,j-1)-E_{I(\rho)_j}, \ 1\le j\le s_\rho.
\end{split}
\end{equation}

The divisor $F(0,0)$ is obtained from the divisor $D'_{(r)}$ as the kernel 
of an exact sequence of type \eqref{seq A}, precisely sequence \eqref{seq Ar}. 
For $1\le \rho\le n-2$ (resp. $\rho=n-1$) each divisor $F(\rho,j)$ is obtained from the previous one
 as kernel of a sequence of type \eqref{seq B} (resp. \eqref{seq C}), cfr. Section \ref{sequences}.
Notice that the last divisor obtained is  $F(r,s_r)=D_{(r)}$.

We claim that, for all pairs $(\rho,j)$ such that $I(\rho)_j\in \mathcal{I}(\rho)_1^>$, 
if $\rho\le n-2$ (resp. $\rho=n-1$),
 the first factor of the restriction (resp. the restriction) of $F(\rho,j-1)$ to $E_{I(\rho)_j}$ in the sequence of type
\eqref{seq B} (resp. \eqref{seq C}) is $-\Cr_\rho(h)$.
Cfr. \eqref{cremona} for a definition of the divisor class $\Cr_\rho(h)$. The latter  has vanishing
 cohomologies by Theorem \ref{vanishing cremona}. Hence, by the Kunneth formula, we conclude that the restriction itself has vanishing cohomologies. 

In order to prove the claim, recall that the first factor of $E_{I(\rho)_j}$ is isomorphic to  $X^{\rho}_{(\rho-2)}$, the toric variety obtained from $\PP^\rho$ by blowing-up all coordinate points, lines etc., cfr. Remark \ref{permutohedron}.
Using \eqref{D'r}, \eqref{F} and the third line in \eqref{recursion}, we can expand the expression of 
$F{(\rho,j-1)}$:
\begin{align*}
F{(\rho,j-1)}=dH&
- \sum_{I(t)\in\II(\rho-1)} k_{I(t)}E_{I(t)}
 -\sum_{1\le l\le j-1} k_{I(\rho)_l}E_{I(\rho)_l}
  -(k_{I(\rho)_j}-1)E_{I(\rho)_j}\\
 & -\sum_{j< l\le s_\rho} (k_{I(\rho)_l}-1)E_{I(\rho)_l}
- \sum_{I(t)\in\II(n-1)^>_1\setminus\II(\rho)} (k_{I(t)}-1)E_{I(t)}.
\end{align*}
Using the intersection table \eqref{intersection table}, we see that the second summation in the first line above, as well as 
the whole second line,  restricts to $0$ on the first factor of 
 $E_{I(\rho)_j}$. Therefore, the restriction of 
$F{(\rho,j-1)}+(k_{I(\rho)_j}-1)E_{I(\rho)_j}$ to ${E_{I(\rho)_j}}$ is the following product
\begin{equation}\label{product}
 \left(dh-
\sum_{\substack{I(t)\subset I(\rho)_j\\ \ I(t)\in\II(\rho-1)}}k_{I(t)}e_{I(t)}\right)\boxtimes\ast
=\left(-k_{I(\rho)_j}\Cr_\rho(h)\right) \boxtimes\ast,
\end{equation}
where $\ast$ denotes the appropriate divisor.
In order to prove the equality in \eqref{product}, first notice that,
for every $I(\rho-1)\in\mathcal{I}$, $I(\rho-1)\subset I(\rho)_j$, we have (cfr. Remark \ref{exc of hyperplanes}):

$$
e_{I(\rho-1)}=h-\sum_{\substack{0\le t< \rho-1\\ \ I(t)\subset I(\rho-1)}}e_{I(t)}.
$$
Therefore, assuming without loss of generality that $I(\rho)_j=\{1,\dots,\rho+1\}$, one can see that the coefficient of $h$ is

\begin{align*}
d-\sum_{\substack{I(\rho-1)\subset I(\rho)_j\\ \ I(\rho-1)\in\II(\rho-1)}} 
k_{I(\rho-1)}& =  d-\sum_{i=1}^{\rho+1}(m_1+\cdots+\check{m_i}+\cdots+m_{\rho+1}-(\rho-1) d)\\
& = d-\rho\sum_{i=1}^{\rho+1}m_i+(\rho+1)(\rho-1)d\\
& =-\rho k_{I(\rho)_j}.
\end{align*}
We leave to the reader to verify that also the coefficients of the exceptional 
divisors in the left hand side and right hand side of \eqref{product} coincide.
Finally, the self-intersection ${E_{I(\rho)_j}}{|_{E(\rho)_{j}}}$ is $-\Cr_\rho(h)$ 
on the first factor, see Lemma \ref{normal bundle}. This concludes the proof.

\end{proof}

\begin{proof}[Proof of Theorem \ref{vanishing for >=n+3}]
We first prove part $(b)$; part $(a)$  follows. 

The proof is by induction on $n$, the dimension of the ambient space,  and on $b=b(D)$.
If $n=1$ the statement is trivially true, as well as if  $b=-nd$, i.e. if $m_i=0$ for all $i=1,\dots,s$.

Fix the pair $(n,b)$ and assume by induction that the statement is true for all divisors in the $(n-1)$-dimensional space 
satisfying the bound \eqref{EffectivityCondition}, and for all divisors with the values $(n,b-1)$.
In particular, the statement is true for the divisors $D'=D+E_1$ and ${D'_{(1)}}{|_{E_1}}$. 
By Lemma \ref{condition b kernel} and Lemma \ref{cond b n+3}, both $D'$ and  ${D'_{(1)}}{|_{E_1}}$ satisfy condition \eqref{EffectivityCondition}. Finally, for all $1\le r\le n-1$, the conclusion follows from Propositions \ref{h^0 exact}, \ref{simple obstructions} and  \ref{sequence B}.

\end{proof}

\section{Non-effective divisors on blown-up projective space}\label{section non-effective statements}

For an arbitrary number of general points $s$, let $D$ be a divisor on $X_{(0)}^n$ of the form \eqref{general divisor}:
$$
D=dH-\sum_{i=1}^{s}m_i E_i,
$$
 with $m_i\leq d+1$ and  $d,m_i\ge0$.
If $m_i=d+1$, for some $i$, then $D$ is obviously not effective, namely it has  no non-zero global sections.
Furthermore, under some further assumptions on the integer $b=\sum_{i=1}^sm_i-nd$, 
introduced in \eqref{b}, we will see in this and in the next section that the same vanishing 
theorems as in the effective case hold for the
cohomology groups of the strict transforms of $D$ in $X^n_{(r)}$.

\subsection{Formal definition of linear base locus for non-effective divisors}\label{section base locus empty case}
With the aim of studying the cohomology of the strict transforms on $X^n_{(r)}$ of non-effective divisors, 
we generalise the notion of linear base locus.

In order to do so, we introduce the following notation, that generalises  \eqref{II}:
\begin{equation}\label{JJ} 
\JJ^{>}:=\JJ^{>}(D):=\{I(r)\subseteq\{1,\dots,s\}:0\le r\le n-1, K_{J}> 0, \forall J\subseteq I(r)\},
\end{equation}
where, given $(d,m_1,\dots,m_s)$, the integer $K_J$ is defined as in Section \ref{linear base locus lemma}.
We have the following inclusion of sets: $\JJ^>(D)\subseteq \II^>(D)$.

\medskip 

For a non-effective divisor, whenever $I(r)\in\JJ^>$, we will say that the cycle $L_{I(r)}$ is \emph{contained  in the base locus of} $D$.
Moreover we will call \emph{strict transform} of $D$ the divisor $D_{(r)}$
 (and $\tilde{D}=D_{(\bar{r})}$) defined as the formal sum of divisors given by 

 \begin{equation}\label{strict transform empty}
 D_{(r)}:=dH-\sum_{\substack{I(\rho)\in\JJ^>\\0\le\rho\le r}} k_{I(\rho)}E_{I(\rho)},
 \end{equation}
where 
 $$\bar{r}=\bar{r}(D):=\max\{r:I(r)\in \JJ^>\}.$$

\begin{remark}
The above definitions formally extend the notion of linear base locus and of strict 
transform, after blowing-up such base locus, from effective divisors to non-effective ones. When $D$ is effective,  we have $m_i\le d$. Therefore the condition $K_J\ge K_{I(r)}$ is obviously satisfied for every $J\subseteq I(r)$. This implies the following equality of sets: $\JJ^>(D)= \II^>(D)$. In particular \eqref{strict transform} and \eqref{strict transform empty} coincide.

Moreover each exceptional divisor $E_{I(\rho)}$ in $X_{(r)}$ will be a product of blown-up projective spaces 
and the first factor will be the toric variety described in Remark \ref{permutohedron}, cfr. Remark \ref{first factor is toric} 
\end{remark}

\subsection{Main theorems}

 We will prove the following formula for the dimension of the  $(r+1)-$st cohomology group of $D_{(r)}$ in terms of the speciality of $\tilde{D}$.
 Recall from \eqref{linear obstruction} that $l(D,r+1)$ 
denotes the alternating sum of the integers \eqref{The Newton binomial} that compute the
contribution of the linear $r$-cycles $L_{I(r)}$ of dimension at least $r+1$, contained in the base locus of $D$ with
multiplicity $k_{I(r)}$, to its speciality.

\begin{theorem}\label{higher cohomologies}
Let $D$  in $X_{(0)}^{n}$, as in \eqref{general divisor}, be any effective divisor or a non-effective divisor with $m_i\leq d+1$. Assume that $\JJ^>(D)$ satisfies conditions \emph{(I)}, \emph{(II)} and \emph{(III)} of Section \ref{blow-up construction}. 
Then the following holds. Let $0\leq r\leq n-1$.
\begin{enumerate}
\item  We have
$$h^{r+1}(D_{(r)})=l(D,r+1)+
\sum_{\rho=r+1}^{n}(-1)^{\rho-r-1}h^{\rho}(\tilde{D}).$$
Moreover, if $h^i(\tilde{D})=0$, for all $i\ge 1$, then $h^i(D_{(r)})=0$ for all $i\ne r+1$.
\item
For any integer $l_{I(r)}$ with  $0\le l_{I(r)}\le\min(r,k_{I(r)})$, we have
\[
h^i(D_{(r)})=h^i(D_{(r)}+\sum_{I(r)} l_{I(r)} E_{I(r)}), \quad 
\mbox{ for $i\ge 0$}.
\] 
\end{enumerate}
\end{theorem}

 Observe first that for $r=-1$ the binomial sum $l(D, 0)$ defined in \eqref{linear obstruction} becomes $\ldim(D)$, the linear virtual dimension of $D$ introduced in Definition \ref{new-definition}. Moreover, setting $D_{(-1)}:=D_{(0)}=D$ for $r=-1$ the theorem reads

\begin{corollary}\label{number of sections}
For any $D$  in $X_{(0)}^{n}$ effective divisor or  non-effective divisor with 
$m_i\leq d+1$ and $\JJ^>$ satisfying conditions \emph{(I)}, \emph{(II)} and \emph{(III)}
 from Section \ref{blow-up construction}, then
 $$h^0(D)=\ldim (D)+\sum_{\rho=1}^{n}(-1)^{\rho}h^{\rho}(\tilde{D}).$$
\end{corollary}

\subsection{Vanishing theorems for toric divisors}

If $s\le n+1$, the blow-up $X^n_{(0)}$ of $\PP^n$ at $s$ points in general position, 
that we can think of as coordinate points of $\PP^n$, is a toric variety. 
All the blow-ups $X^n_{(r)}$ of $X^n_{(0)}$ in this case will also be toric varieties, cfr. Remark \ref{permutohedron}.
We will call \emph{toric divisor} a divisor $D$ on $X_{(0)}^n$ of the form 
\eqref{general divisor}, not necessarily effective, with $s\le n+1$. 
This section is devoted to the study of the cohomology of the
strict transforms $D_{(r)}$ of $D$, for $D$ a toric divisor.

\begin{theorem}\label{toric theorem}
If $D$ is any toric divisor with $m_i\leq d+1$, then $h^0(D)=\ldim(D)$, $h^{n}(\tilde{D})={{b-1}\choose{n}}$ and $h^{i}(\tilde{D})=0$ for every $1\leq i\leq n-1$.
\end{theorem}

\begin{remark}
If $D$ is an effective toric divisor (therefero $m_{i}\leq d$ and
 $b(D)\leq 0$), then it is $h^{n}(\tilde{D})=0$. In this case Theorem \ref{toric theorem} 
is just a particular case of {\cite[Theorem 4.6]{BraDumPos}}. However, for  all  non-effective toric divisors with $m_i\leq d+1$ 
this result is new. Theorem \ref{toric theorem} suggests that the \emph{virtual} $n$-dimensional 
cycle $L_{1,\dots,n+1}\cong\PP^n$ should be considered in the dimension count. 
More precisely, this virtual cycle is detected by the $n-$th cohomology group; its contribution depends on its virtual multiplicity $k_{1,\ldots,n+1}$ that becomes nothing else than $b(D)$. 
Moreover, for non-effective toric divisors with $m_1=d+1$ or $b(D)\leq 0$, Theorem \ref{toric theorem} implies that $\ldim (\LL)=0$ where $r$ runs from $-1$ to $n$ (see also Corollary \ref{binomial equalities}).
\end{remark}

\subsubsection{Chambers of non-effective toric divisors with vanishing theorems}\label{FI-chi}

We recall that the effective cone $\Eff(X)$ of $X^n_{(0)}$, for $s\le n+1$, 
is given by the inequalities 
\begin{equation}\label{toric effective cone}
d\ge0,  \quad b\le0, \quad m_i\le d, \quad \forall i\in\{1,\dots,s\}.
\end{equation}
See e.g. \cite[Lemma 2.2]{BraDumPos}.
Moreover, for any effective toric divisor, the strict transform $\tilde{D}$ 
in $X^n_{(n-1)}$ has
 vanishing cohomologies, see Theorem 
\ref{monster for effective}.

The following inequalities define chambers of  $\N^1(X)\setminus \Eff(X)$ such that if $D$
lies in those chambers, then
$\tilde{D}$ has vanishing theorems, by Theorem \ref{toric theorem} (cfr. also Theorem \ref{general statement vanishings}):
\begin{subequations}\label{bounds for multiplicities 1}
\begin{align}
s\le n+1:& \quad 1\le b\le n, \quad m_j\le d+1 \label{n+1 b pos};\\
s\le n+1:& \quad  b\le 0, \quad m_1=d+1,\  m_{j}\le d+1 \label{n+1 b neg}.
\end{align}
\end{subequations}

 \begin{remark}\label{(I)(II)(III) verified for toric chambers}
  In the above notation, if $D$ satisfies \eqref{bounds for multiplicities 1}, then the set $\JJ^>(D)$ satisfies conditions  (I), (II) and (III) of Section \ref{blow-up construction}.
 \end{remark}

Theorems \ref{higher cohomologies} and \ref{toric theorem} have the following consequence, for divisors in the chambers \eqref{bounds for multiplicities 1}.

\begin{corollary}\label{toric corollary}
Statements $(a)$ and $(b)$ of Theorem \ref{monster for effective}  hold for non-effective toric divisors  satisfying the inequalities 
\eqref{bounds for multiplicities 1}.
\end{corollary}

\subsubsection{Strict transforms and Cremona transformations}

In this section we show that the strict transform $\tilde{D}$ of a non-effective toric divisor equals a negative multiple of the strict transform of the standard Cremona transformation of the hyperplane class of $\PP^n$. Section \ref{Cremona of hyperplane class} was dedicated to a cohomological descriptions of the latter.

\begin{proposition}\label{strict=cremona}
Let  $D$ be a toric divisor with $d\ge 1$, $0\le m_i\le d+1$.
\begin{enumerate}
\item
If $0\le m_i\le d$ and $b=0$, then $D_{(n-1)}=0$.
\item
If either $b\geq n$ or  $m_{i}\leq d$ for all $i$ and $b\ge1$, then $D_{(n-1)}=-b\Cr_n(H)$.
\end{enumerate} 
\end{proposition}
\begin{proof}
Notice in the first case that $b=0$ for $s\leq n$, can occur only for $m_i=d$, for all $i\in \{1,\ldots,n\}$. In this case, the strict transform is obviously $D_{(n-1)}=0$, because it is obtained by subtracting the strict transform of the hyperplane through the $n$ points $d$ times. This observation also implies in the second case that the hypothesis forces $s=n+1$. This allows us to restrict in both cases to $s=n+1$ for the rest of the proof.
\\

Case (1) if $s=n+1$ with $b=0$ and $m_j\le d$, for all $j$, then all hyperplanes spanned
 by sets of $n$ points 
 are contained in the base locus of $D$ with (exact) multiplicity 
$$k_{1,\dots, \check{j},\dots,n+1}=K_{1,\dots, \check{j},\dots,n+1}=\sum_{i=1}^{n+1} m_{i}-m_j-(n-1)d=
d-m_j\ge 0.$$

Similarly, for Case (2), if $s=n+1$ and $m_j\le d+1$ for all $j$, the assumption $b\ge 1$ implies that
$$k_{1,\dots, \check{j},\dots,n+1}=K_{1,\dots, \check{j},\dots,n+1}=\sum_{i=1}^{n+1} m_{i}-m_j-(n-1)d=
b+d-m_j\ge 0.$$

Moreover, if $m_i\leq d$ and $K_{I(r)}> 0$, then for any subset $I(\rho)\subset I(r)$ we have $K_{I(\rho)}> 0.$ 
The same holds for $m_i\leq d+1$ and $b\geq n$ , indeed these two assumptions together imply that $K_{I(r)}\ge r \ge0$, for every $I(r)$ with $1\le r\le n-1$. 
Therefore, also in this case, all hyperplanes spanned by sets of $n$ points  are contained in the base locus of $D$ 
with (exact) multiplicity  $k_{I(n-1)}$.

In all above cases, the strict transform $D_{(n-1)}$ of $D$ is
$$
dH
-\sum_{I(\rho)\in\II(n-2)}k_{I(\rho)}E_{I(\rho)} 
-\sum_{\substack{I(\rho)\in\JJ^>\\0\le\rho\le n-2}} k_{I(\rho)}E_{I(\rho)} 
-\sum_{i=1}^{n+1}k_{I(n-1)}E_{I(n-1)},
$$
where $E_{I(n-1)}$ is the strict transform in $X^n_{(n-1)}$ of the hyperplane of $\PP^n$
passing through the $n$ points parametrised by $I(n-1)$, see 
Remark \ref{permutohedron}. 
One can verify that the coefficients of the hyperplane class $H$ and the coefficient of the exceptional divisors $E_{I(\rho)}$ in the  
expression for $D_{(n-1)}$ are respectively 
\begin{align*}
d-\sum_{i}(b-m_i+d)=&-nb,\\
k_{I(\rho)}-\sum_{ i\notin I(\rho)}(b-m_i+d)=&-(n-\rho-1)b.
\end{align*}
This concludes the proof.
\end{proof}

\begin{remark}
Proposition \ref{strict=cremona} provides another proof that the effective cone of $X^n_{(0)}$, the 
space blown-up at $s\leq n+1$ points, is described  by the inequalities \eqref{toric effective cone}.
\end{remark}

\subsection{Vanishing theorems for non-effective divisors with $s=n+2$ points or more}

We recall here that in the case $s=n+2$, the space  $X^n_{(n-2)}$ was identified by Kapranov \cite{Ka}
with the moduli space $\overline{\mathcal{M}}_{0,n+3}$ of stable rational curves with $n+3$ marked points.
There are chambers of divisors with $s=n+2$ 
for which the strict transform $\tilde{D}$ in $X^n_{(n-1)}$ has vanishing theorems. 

\subsubsection{Chambers of non-effective divisors on  $\overline{\mathcal{M}}_{0,n+3}$ with vanishing theorems}

We recall that the effective cone $\Eff(X)$ of $X^n_{(0)}$, with $s=n+2$, is formed by  the divisors  of the form \eqref{general divisor} that satisfy the following inequalities:
\begin{equation}\label{n+2 effective cone}
d\ge0, \quad b\le0, \quad m_i\le d, \quad b\le m_i,  \quad \forall i\in\{1,\dots,n+2\},
\end{equation}
where  $b:=b(D)=\sum_{i=1}^{n+2}m_i-nd$, as defined in \eqref{b}.
See e.g. \cite[Lemma 2.2]{BraDumPos}.
Moreover, for any effective divisor $D$ satisfying \eqref{n+2 effective cone}, the strict transform $\tilde{D}$ has
 vanishing cohomologies, see Theorem \ref{monster for effective}.

In this section  we extend the vanishing theorems to  
divisors sitting 
in a ``small'' region outside the effective cone of $X:=X^n_{(0)}$, the blow-up  of $\PP^n$ at $ n+2$ points in general position.

For a divisor $D$ of the form \eqref{general divisor}, 
we consider the following chambers of  $\N^1(X)\setminus \Eff(X)$:
\begin{subequations}\label{bounds for multiplicities 2}
\begin{align}
s=n+2:& \quad b=1, \quad m_j\le d+1  \label{n+2 b=1};\\
s= n+2:& \quad  b\le 0, \quad m_1=d+1,\ m_j\le d+1. \label{n+2 b neg}
\end{align}
\end{subequations}

 \begin{remark}\label{(I)(II)(III) verified for n+2 chambers}
 In the above notation, if $D$ satisfies \eqref{bounds for multiplicities 2}, then the set $\JJ^>(D)$ satisfies conditions (I),(II) and (III) of Section \ref{blow-up construction}.
 \end{remark}

\begin{theorem}\label{monster for empty}
Statements $(a)$ and $(b)$ of Theorem \ref{monster for effective}  hold for non-effective divisors with $s= n+2$ points
 satisfying the inequalities 
\eqref{bounds for multiplicities 2}.
\end{theorem}

We will also generalise Theorem \ref{vanishing for >=n+3} 
to non-effective divisors  with an arbitrary number of points in a small region around the effective cone, namely for $m_i\leq d+1$.
Recall the bound \eqref{EffectivityCondition} of Theorem \ref{theorem a n+3} that was proved in \cite{BraDumPos} to be, 
together with $m_i\le d$, sufficient condition
for a divisor with $s\ge n+3$ to be effective. 

\begin{theorem}\label{empty arbitrary points} 
Statements $(a)$ and $(b)$ of  Theorem \ref{monster for effective}  hold for non-effective divisors with $s\ge n+3$ 
points such that $m_i\leq d+1$ and $b\le \min(n-s(d),s-n-2)$. 
\end{theorem}

\subsection{Euler characteristic and linear virtual dimension}

In this section we compute the Euler characteristic of the strict transforms
 of non-effective divisors $D$ and compare them with their linear expected dimension.
 
\begin{corollary}\label{binomial equalities}
Let $D$ be a divisor with $m_i\le d+1$. The following holds.
\begin{enumerate}
\item
If $D$ is toric and effective, then 
$$\chi(\tilde{D})=\ldim(D)\ge1.$$
In particular if $b=0$, then 
$$\chi(\tilde{D})=\ldim(D)=1.$$
\item If $D$ is toric and non-effective, then 
$$\ldim(D)=0, \quad \chi(\tilde{D})=(-1)^nh^n(\tilde{D})=(-1)^n{{b-1}\choose n}.$$
In particular if $b\le n$ then 
$$ \chi(\tilde{D})=h^n(\tilde{D})=\ldim(D)=0.$$
\item If $D$ is non-effective  
such that  $s\le n+1$ and $b\le n$, 
or $s=n+2$ and $b\le 1$,
or $s\ge n+3$ and $b\le \min(n-s(d),s-n-2)$, then
$$\chi(\tilde{D})=\ldim(D)=0.$$
\end{enumerate}
\end{corollary}

\begin{proof} 
Part (1)  follows from Proposition \ref{strict=cremona} and {\cite[Theorem 4.6]{BraDumPos}}. Moreover, 
part (2) and (3) follow from Theorem \ref{toric theorem} if $s\le n+1$, 
Theorem \ref{monster for empty} if $s=n+2$ and Theorem \ref{empty arbitrary points} for $s\ge n+3$, in fact $h^0(D)=\ldim(D)=0$.
\end{proof}

In particular all non-effective divisors satisfying the conditions of Corollary \ref{binomial equalities}(3) are only linearly obstructed and their strict transform $\tilde{D}$ is not linearly obstructed.

We conclude the section by noticing that
if one moves further away from the effective cone, by choosing for instance
$m_1=d+2$ or higher, then the strict transform after blowing-up  all linear subspaces contained
in the base locus may still be linearly obstructed. 
We illustrate some instances where this 
happens.

\begin{example}
Let $D=3H-5E_1-mE_2$ in $X^3_{(0)}$ be a divisor  with $h^0(D)=0$. 
\begin{itemize}
\item
If $m=3$, we have $\chi(D_{(1)})=-5$. Hence $h^1(D_{(1)})\ge5$.
\item
If $m=4$, we have $\chi(D_{(1)})=0$.  Is  $h^1(D_{(1)})=h^2(D_{(1)})=0$?
\item
If $m=5$, we have $\chi(D_{(1)})=6$. Hence $h^2(D_{(1)})\ge6$.
\end{itemize}
\end{example}

\section{Proofs of the result of Section \ref{section non-effective statements}}\label{section non-effective proofs}

In this section we prove all results stated in Section \ref{section non-effective statements}.
More precisely this section is organised as follows. 
We prove Theorem \ref{toric theorem}, Theorem \ref{monster for empty} and Theorem \ref{empty arbitrary points} 
for the planar case in Section \ref{section planar case}. This case is the base step for an induction argument on $n$. 
In Section \ref{section induction procedure} we will construct the induction procedure, that will follow the line of Section \ref{Vanishing theorems section}. 
We will complete the proof of Theorem  \ref{toric theorem}, Theorem \ref{higher cohomologies}, Theorem \ref{monster for empty} and Theorem \ref{empty arbitrary points}  for dimension $n\ge3$ in Sections \ref{Section toric case}, \ref{section proof higher cohomologies},
\ref{section case n+2} and \ref{section case n+3} respectively. 

\subsection{The planar case}\label{section planar case}

\begin{proposition}\label{planar vanishings}
Let $D$ be a non-effective divisor on $X_{(0)}^2$  with $m_i\le d+1$ such that $b\le 2$ if $s\le3$, $b\le 1$ for $s=4$ and $b\le \min(n-s(d),s-4)$ for $s\ge 5$. 
Then $h^i(\tilde{D})=0$, for $i=0,1,2$.
\end{proposition}

\begin{proof}
If $s=1$ the statement is trivial. Therefore we will assume that  $s\ge2$. Without loss of generality, we will also assume $m_1\ge\cdots\ge m_s\ge1$. 

Notice that in the case $s\le3$ (resp. $s=4$), the assumption that $D$ is not effective implies that \eqref{bounds for multiplicities 1} (resp \eqref{bounds for multiplicities 2}) is satisfied. Moreover if $s\ge 5$, the same assumption implies that $m_1=d+1$.

Assume $s\le3$ and $b=2$. In this case the strict transform of $D$ is $\tilde{D}=-2\Cr_2(H)$ and $h^i(\tilde{D})=0$ for $i=0,1,2$ by Theorem \ref{vanishing cremona}.

\medskip 

We will assume from now on that $b\le1$ for $s\le4$ and $b\le s-4$ for $s\ge5$.

\medskip
Case (1) $m_1=d+1$. We have  $k_{1i}=1+m_{i}\geq 1$, for $2\leq i\leq s.$
The condition $b=(d+1)+m_2+\cdots+m_s-2d\leq 1$  implies that $k_{ij}\leq 0$, if $1\notin\{i, j\}.$
We can write
$$\tilde{D}=D_{(1)}:=D-\sum_{i=2}^{s}k_{1i}(H-E_{1}-E_{i})=(-b-s+2)H-(-b-s+3)E_{1}+\sum_{i=2}^{s}E_{i}.$$
By Serre duality and the fact that $b\le 1$, we obtain 
$$h^2(\tilde{D})=h^0((b+s-5)H-(b+s-4)E_{1})=0.$$ 
Moreover, using $\chi(\tilde{D})=0$ and $h^0(\tilde{D})=0$, we  conclude that also $h^1(\tilde{D})=0$.

\medskip 

Case (2) $m_1\le d$.  In this case $s\le 4$ and the non-effectivity assumption implies that $b=1$.

If $s=3$, then $m_1+m_2+m_3-2d=1$ so all three lines are in the base locus of $D$. 
We leave to the reader to check  that $\tilde{D}=-2H+E_1+E_2+E_3$. 
The vanishing of $h^i(\tilde{D})$, for $i=0,1,2$, follows from Theorem \ref{vanishing cremona}.

Assume $s=4$ and $m_{1}=d$. We have $k_{1i}=m_{i}$ and we conclude $h^{1}(\tilde{D})=h^{2}(\tilde{D})=0,$ since the strict transform is
$$\tilde{D}=-H+E_{1}.$$

Case $s=4$ and $m_{1}<d.$ We must have $k_{12}=m_1+m_2-d>0$ (otherwise 
$K_{34}\leq K_{12}\leq 0$ would be in contradiction with $b\ge1$). From $K_{ij}+K_{pq}=1$, we obtain that $k_{34}=0$ and that
the lines $L_{ij}$ and $L_{pq}$ can not be  
both contained in the base locus of $D$, therefore either $k_{23}=1$ or $k_{14}=1$. 
Now, observe that by assumption we have $k_{13}\geq k_{23}$ and $k_{13}\geq k_{14}$. We have the following two cases.

 Case (i). The lines $L_{12}, L_{14}, L_{13}$ are in the base locus and
$$\tilde{D}=[4d-(3m_1+m_2+m_3+m_4)]H-[3d-(2m_1+m_2+m_3+m_4)]E_1-\sum_{i=2}^4(d-m_1)E_i.$$

Case (ii).  The lines $L_{12}, L_{23}, L_{13}$ are in the base locus  and
$$\tilde{D}=2(m_4-1)H-\sum_{i=1}^3 (m_4-1)E_i-m_4E_4.$$

We claim $h^i(\tilde{D})=0, i\geq 0$. Indeed, in both cases $\tilde{D}$ has the property that
$k_{12}(\tilde{D})=0$ and $k_{34}(\tilde{D})=1$, therefore $h^i(\tilde{D})=h^i(\tilde{D}-L_{34})$. 
Furthermore, the divisor $\tilde{D}-L_{34}$ has now $k_{12}=1$ and $k_{34}=0$, 
therefore $h^i(\tilde{D}-L_{34})=h^i(\tilde{D}-L_{34}-L_{12})$. We continue to eliminate simple obstructions until the residue becomes a toric divisor with positive coefficients of the form $H-E_1-E_2-E_3$ that has all vanishing theorems.
\end{proof}

\begin{proposition}\label{n=2} 
Statements $(a)$ and $(b)$ of Theorem \ref{monster for effective}  hold for non-effective divisors in $\PP^{2}$ 
  satisfying the inequalities \eqref{bounds for multiplicities 1} if $s\le 3$,  \eqref{bounds for multiplicities 2} if $s=4$ and $b\le \min(n-s(d),s-4)$ if $s\ge 5$. 
\end{proposition}

\begin{proof} For  non-effective divisors in $\PP^{2}$, we reduce the
 proof to the vanishing theorems for $\tilde{D}$. Indeed, for any line $L_{ij}$ through two points $p_i$ and $p_j$ with corresponding $k_{ij}:=-D\cdot (H-E_i-E_j)\geq 1$, 
we have $D|_{L_{ij}}=\mathcal{O}_{\PP^{1}}(-k_{ij})$.

For $\tilde{D}=D-\sum_{1\leq i\leq j\leq 4}k_{ij}(H-E_i-E_j)$, the Riemann-Roch Theorem implies
\begin{equation}\label{rr}
\chi(\tilde{D})=\sum_{k_{i,j}\geq 1}{{k_{ij}}\choose{2}}+\chi(D).
\end{equation}

Furthermore, since $D$ has positive coefficients it follows that $h^2(D)=0$ while $h^0(D)=h^0(\tilde{D})$.
The formula \eqref{rr} implies that
\begin{align*}
h^1(D)&=\sum_{k_{ij}\geq 1}{{k_{ij}}\choose{2}} + h^1(\tilde{D}) - h^2(\tilde{D}),\\
\chi(D)&=h^0(D)-h^1(D)={d+2\choose{2}}-\sum_{i=1}^4{m_i+1\choose{2}}.
\end{align*} 
Finally, the above equalities imply
$$h^{0}(D)=\ldim(D)+h^1(\tilde{D})-h^2(\tilde{D}).$$
We conclude using Proposition \ref{planar vanishings}.

\end{proof}

\begin{corollary}
Theorem \ref{toric theorem}, Theorem \ref{monster for empty} and Theorem \ref{empty arbitrary points} 
hold for $n=2$.
\end{corollary}

\subsection{Case $n\ge3$: an induction procedure}\label{section induction procedure}
We will prove the statements of Section \ref{section non-effective statements} by induction on $n$ and on $b$.

As in Section \ref{Induction on b}, let us introduce the following divisor
$$
D':=D+E_1.
$$
We will consider $D'$ and its strict transforms $D'_{(r)}$ as living in 
$X^n_{(r)}$, the space blown-up along the base linear cycles of $D$, parametrised by the set
$\JJ^>(D)$, cfr.  Remark \ref{D' lives in spaces of D}. 
\begin{remark}
Since $m_i\le d+1$, the divisor ${D'_{(1)}}{|_{E_1}}$ in $E_1\cong X^{n-1}_{(0)}$ (introduced in \eqref{D'r}) has 
 $k'_{1i}:=k_{1i}(D')=\max(K_{1i}-1,0)\le m_1$, for all $2\le i\le s$, namely the multiplicities 
of the points $e_{{1i|1}}$
  do not exceed the multiplicity $m_1-1$ by more than one. 
\end{remark}

\begin{lemma}\label{b in seq A} In the above notation,
assume $D$ satisfies \eqref{bounds for multiplicities 1}
if $s=n+1$, or $m_1=d+1$ if $s\ge n+2$. 
Then $b({D'_{(1)}}{|_{E_1}})=b(D')=b-1$. 
\end{lemma}

\begin{proof}
Assume first that $m_1=d+1$. Notice that $k'_{1i}=m_1+m_i-d-1=m_i\ge1$, 
for all $2\le i \le s$. Therefore,
the restricted divisor ${D'_{(1)}}{|_{E_1}}$ is of the form
$$dh-\sum_{i=2}^{s}m_ie_{1i|1},$$ 
and one computes 
$$b(D'_{(1)}{|_{E_1}})=\sum_{i=2}^{s}m_i-(n-1)d=\sum_{i=1}^{s}m_i-nd-m_1+d=b-1.$$

Assume now that $s\le n+1$ and
that $m_i\le d$, for all $1\le i\le s$. In this case $b\ge1$.
Fix an index $i\ge 2$ and write 
$$b-1=\sum_{j=2,\ j\ne i}^{s}m_j-(n-1)d+(m_1+m_i-d-1).$$ 
As $b\ge1$, $k'_{1i}=m_1+m_i-d-1\ge 0$. Therefore, 
the divisor $D'_{(1)}{|_{E_1}}$ is of the form 
$$(m_1-1)h-\sum_{i=2}^{s}k'_{1i}e_{1i|1},$$ with possibly some of the $k'_{1i}$
being zero. One computes 
$$b(D'|_{E_1})=\sum_{i=2}^{s}(m_1+m_i-d-1)-(n-1)(m_1-1)=b-1.$$
\end{proof}

In the conditions of Lemma \ref{b in seq A},
it is clear that the map
$$
H^0(D'_{(r)})\to H^0 ({D'_{(r)}}{|_{E_1}})
$$
 between global sections is injective, as the kernel $D'_{(r)}-E_1$ has no non-zero sections. 
We prove in the next proposition  that 
it is in fact an isomorphism.

\begin{proposition}\label{h^0 in sequence A}
Assume $D$ satisfies the same assumptions as Lemma \ref{b in seq A}.
Then  $h^0(D'_{(1)})=h^0({D'_{(1)}}{|_{E_1}})$.
\end{proposition}

\begin{proof} 
Notice that $h^0(D')=h^0(D'_{(1)})$.
We first consider the case \eqref{n+1 b pos} with $m_i\le d$.
Notice that ${D'_{(1)}}{|_{E_1}}$ has base points with multiplicity bounded by the degree,
 that is $k'_{1i}\le m_1-1$, for all $2\le i\le s$. 
If $2\le b\le n$, then, using Lemma \ref{b in seq A}, 
we get that $b(D')=b({D'_{(1)}}{|_{E_1}})\ge1$, therefore both divisors $D'$ and 
${D'_{(1)}}{|_{E_1}}$ have no non-zero global sections. 
Indeed, they both lie outside the effective cones of the respective spaces,   described in \eqref{toric effective cone}.
If $b=1$, then, using Lemma \ref{b in seq A}, 
we get that $b(D')=b({D'_{(1)}}{|_{E_1}})=0$, so the two divisors are effective. 
Using Proposition \ref{strict=cremona},
 we obtain that they both have only one non-zero global section, namely $h^0=1$.

We now consider the case in which $m_1=d+1$.
Notice that $k'_{1i}=m_i$, for $2\le i\le s$. 
Assume first that  $m_2=d+1$. 
Then $D'$ has a point, $E_2$, with multiplicity larger than the degree, that is $m_2=d+1$; moreover
${D'_{(1)}}{|_{E_1}}$ has a point, $e_{12|1}$, with multiplicity larger than the degree,
 that is  $k'_{12}=d+1$. Therefore, none of them has non-trivial global sections.
If, instead, $m_i\le d$, for all $2\le i\le s$, then both $D'$ and ${D'_{(1)}}{|_{E_1}}$ 
have points of multiplicity bounded by the degree. By the trivial observation that
 $D'$ has a point of multiplicity $d$ and $s-1$ points of multiplicity respectively 
$m_2,\dots, m_s$ and ${D'_{(1)}}{|_{E_1}}$ has $s-1$ points of multiplicity 
respectively $m_2,\dots, m_s$, we conclude that their spaces of global section
have the same dimension, see for instance 
\cite[Lemma 5.1]{BraDumPos}.
\end{proof}

\begin{proposition}\label{vanishing r<n+2 empty}
Assume $D$ satisfies the same assumptions as Lemma \ref{b in seq A} and let $1\le r\le n-1$.

Assume that statement $(a)$ of Theorem \ref{monster for effective}  holds for 
$D'$ and ${D'_{(1)}}{|_{E_1}}$ and that  statement $(b)$ of Theorem  \ref{monster for effective}  holds for the pairs $(D',r)$  and  
$({D'_{(1)}}{|_{E_1}},r-1)$.
Then statements $(a)$ and $(b)$ hold for $(D,r)$.
\end{proposition}

\begin{proof}
We prove the statement in two steps. 
Following the idea of the proofs of Proposition \ref{simple obstructions} and Proposition \ref{sequence B}, we first show that the statement holds for $D'_{(r)}-E_1$ and then that $h^i(D_{(r)})=h^i(D'_{(r)}-E_1)$, for all $i\ge 0$. 

To prove the first part, we consider the sequences in cohomology associated with the short exact sequence  \eqref{seq Ar}, 
as in the proof of Proposition \ref{simple obstructions}. Since $h^0(D'_{(r)}-E_1)=0$,  using \eqref{seq h0}, 
Proposition \ref{h^0 in sequence A} and the assumption 
$h^1(D'_{(r)})=0$ implies  that $h^1({D'_{(r)}}-E_1)=0$.
Moreover, we can compute the other cohomologies of $D'_{(r)}-E_1$ exploiting those of $D'_{(r)}$ 
and ${D'_{(r)}}{|_{E_1}}$ using the long exact sequence in cohomology associated.
In fact if $i\ne 0,r+1$, then $h^i({D'_{(r)}}-E_1)=0$ because $h^{i-1}({D'_{(r)}}{|_{E_1}})=0$ and $h^i(D'_{(r)})=0$ by
induction on $n$ and $b$, respectively. 
If $i=r+1$ we conclude using \eqref{seq hr}.

To prove the second part, we argue as in the proof of Proposition \ref{sequence B}.
If $r=0$ the statement is obvious. 
Fix $1 \le \rho\le r$. 
Similarly to Section   \ref{Induction on b} 
we introduce the sets
$$
\JJ(\rho)^>:=\JJ^>\cap \II(r)_1,
$$
the set of multi-indices that parametrise cycles through $p_1$
 that are in the base locus of $D$.

In the space $X^n_{(r)}$, for every pair $(\rho,j)$, $1\le\rho\le r$ and $0\le j\le s_\rho$, for which $I(\rho)_j\in \mathcal{J}(\rho)_1^{>}$, we consider the divisor $F(\rho,j)$ defined in \eqref{F}, using sequences of type \eqref{seq B}.
We claim that for any pair $(\rho,j)$ such that $I(\rho)_j\in \mathcal{J}(\rho)_1^{>}$,
 the divisor $F(\rho,j)$ restricts, on the first factor of $E_{I(\rho)_j}$  which is isomorphic to $X^{\rho}_{(\rho-1)}$,
to  $-\Cr_\rho(h).$
The latter has vanishing cohomology groups, by Theorem \ref{vanishing cremona}.
Then, by means of the Kunneth formula, we can conclude that the restriction itself has vanishing cohomologies. 
\end{proof}

\subsection{The toric case}\label{Section toric case}

In this section we complete the proof of Theorem \ref{toric theorem}, for $n\ge3$. The case $n=2$ was 
covered in Section \ref{section planar case}.

\begin{proof}[Proof of Theorem \ref{toric theorem}] In this proof $D$ will be a toric divisor. 
If $D$ is effective, i.e. when $b\leq 0$ and $m_i\leq d$ (cfr. \eqref{toric effective cone}, the vanishing theorems for the higher cohomology groups of the strict transforms of $D$ in $X_{(r)}$ were established in Theorem \ref{monster for effective}.

Assume $\bar{r}= n-1$ and consider the following two independent cases:
\begin{enumerate}
\item $b\geq n$ or $b\ge 1$ and  $m_i\leq d$ for all $1\leq i\leq n+1$. 
\item $b \leq n-1$ and $m_1=d+1$. 
\end{enumerate}

Case (1). By Proposition \ref{strict=cremona} part (2), we conclude that
$\tilde{D}=-b\Cr(H).$
In Theorem \ref{vanishing cremona} we proved that for all $i\neq n$ then $h^i(\tilde{D})=h^i(\mathcal{O}(-b\Cr_n(H)))=0$, while
$$h^n(\tilde{D})=h^n(\mathcal{O}(-b\Cr_n(H)))=h^n(\PP^n, \mathcal{O}(-b))={{b-(n+1)+n}\choose{n}}.$$
In particular, if $b=n$, we obtain $h^i(\tilde{D})=0, \forall i\geq 0.$

Case (2). We assume that $1\le b\le n-1$. In this case, following the notation of Section
 \ref{section induction procedure}, we have, by Lemma \ref{b in seq A}
that  $b({D'_{(1)}}{|_{E_1}})=b(D')=b-1$.  
Therefore by Proposition \ref{strict=cremona}, we have 
that the two divisors have  strict transforms equal to $-(b-1)\Cr_{n-1}(h)\in 
\Pic(E_1)\cong\Pic(X^{n-1}_{(n-3)})$ and 
$-(b-1)\Cr_{n}(H)\in \Pic(X^{n}_{(n-2)})$, hence both have vanishing cohomologies.
By means of the long exact sequence in cohomology associated to 
the sequence of type \eqref{seq A}, used in \eqref{seq hr} with $r=n-1$, 
we obtain $H^i(\tilde{D})=0$, for every $i\ge0$.
Using the formulas \eqref{from tildeD to Dr}, we can deduce the cohomologies of
the strict transforms $D_{(r)}$, for $1\le r\le n-2$.

\medskip

Assume $\bar{r}\le n-2$. 
The proof is by induction on $b$  and $n$. 
If $n=2$ the statement is covered in Section \ref{section planar case}. The case  $b=-nd$, i.e. if $m_i=0$ for all 
$i=1,\dots,s$ is trivial.

For the pair $(n,b)$, with $n\ge 3$, and $m_1\ge 1$, we will 
assume the statement to be true for $(n-1,b)$ 
and $(n,b-1)$. We can conclude using Proposition \ref{vanishing r<n+2 empty}
that provides the induction step, that $h^i(D_{(\bar{r})})=0$, for all $i\ge0$. 
Using the formulas \eqref{from tildeD to Dr}, we can deduce the cohomologies of
the strict transforms $D_{(r)}$, for $1\le r\le \bar{r}-1$.

\end{proof}

\subsection{Proof of Theorem  \ref{higher cohomologies}}\label{section proof higher cohomologies}

\begin{proof}[Proof of Theorem \ref{higher cohomologies}]
Let $D$ be a divisor of the form \eqref{general divisor}.
Assume that $k_{I(r)}=k_{I(r)}(D)\geq 1$ for some linear cycle of dimension $r$, $L_{I(r)}$, spanned by a subset $I(r)$ on $r+1$ points.

Consider the toric divisor \begin{equation}\label{g}
G=dh-\sum_{i\in I(r)}m_ie_i
\end{equation}
on $X^{r}_{(r-2)}$. Notice that $k_{I(r)}=b(G)$ and $m_i\leq d+1$.

We have that the restriction of $D_{(r-1)}$ to the exceptional divisor is of the form
$$D_{(r-1)}|_{E_{I(r)}}=\tilde{G}\boxtimes 0,$$
where $\tilde{G}$ represents the strict transform of the divisor $G$ defined in \eqref{g}, cfr.  \eqref{intersection table}.

By applying 
Theorem \ref{toric theorem}
 to the toric divisor $G$ in $X^r_{(0)}$  with $m_{i}\leq d+1$ and $b(G)=k_{I(r)}$, we obtain
\begin{equation}\label{bdp2}
h^i(E_{I(r)}, D_{(r-1)}|_{E_{I(r)}})=h^i(X^r_{(r-2)},\tilde{G})=h^i(X^r_{(0)},-k_{I(r)}h).
\end{equation}

Furthermore, since conditions (I),(II) and (III) of Section \ref{blow-up construction} are satisfied and using \eqref{bdp2}, the result contained in  \cite[Proposition 4.10]{BraDumPos} implies that the following formulae also hold
\begin{equation}\label{bdp}
h^{r+1}(D_{(r)})=h^{r+1}(D_{(r+1)})- h^{r+2}(D_{(r+1)})+
\sum_{k_{I(r)\geq 1}}{{n+k_{I(r)}-(r+1)-1}\choose{n}},
\end{equation}
while, for all $i\le r+1$, we have
\begin{equation}\label{bdp1}
h^{i}(D_{(r+1)})=h^{i}(D_{(r+2)}).
\end{equation}

We recall that, if $r=n-1$, i.e. $I$ is a subset with $|I|=n$, then \eqref{bdp} and \eqref{bdp1}  hold for the divisor $E_{I}=\tilde{H}_{I}$, 
a blown-up hyperplane along the cycles spanned by the subsets of $I\in \JJ^{>}$ 
in $X_{(n-3)}^{n-1}$. Notice that for each multi-index $I\in \JJ^{>}$ of cardinality $n$ and $H_{I}$ hyperplane spanned by 
the points parametrised by $I$, with $k_{I}\geq 1$, then on $X_{(n-2)}^{n}$ we have the following divisor:
$$\tilde{D}=D_{(n-2)}-\sum_{I=I(n-1)\in \JJ^{>}}k_{I}H_{I}.$$

We iterate the formula \eqref{bdp} 
for $h^{\rho+1}(D_{(\rho)})$ with $\rho\ge r+1$. More explicitly, \eqref{bdp1} allows us to substitute $h^{\rho+1}(D_{(\rho+1)})=h^{\rho+1}(D_{(\rho+2)})=\cdots=h^{\rho+1}(\tilde{D})$ and apply again equation
\eqref{bdp} for $h^{\rho+2}(D_{(\rho+1)})$. By iteration of \eqref{bdp} for $\rho\ge r+1$ we obtain that
 $h^{r+1}(D_{(r)})$ equals the sum of $\sum_{\rho=r+1}^n (-1)^{\rho-r-1}h^{\rho}(\tilde{D})$ and $l(D, r+1)$, see \eqref{linear obstruction}. This concludes the proof of part (1).

To prove part (2), observe that $D_{(r)}$ is obtained from $D_{(r-1)}$ by subtracting 
 $k_{I(r)}$ times the exceptional divisors $E_{I(r)}$, for all $I(r)$ on the space $X^{n}_{(r)}$. 
By the above argument we obtain that the restricted divisor, for $l=0,\dots,k_{I(r)}-1$, satisfies 
$$h^i(E_{I(r)}, D_{(r-1)}-(k_{I(r)}-l)E_{I(r)}|_{E_{I(r)}})=h^i(\PP^{r},-lh).$$
The right-hand side cohomology group is zero, for all $0\le l\le \min(r,k_{I(r)})$.
\end{proof}

\subsection{The case $s= n+2$}\label{section case n+2}

\begin{remark}
We say that a divisor of the form \eqref{proper transform 1} has \emph{non-negative coefficients} if and only if the degree $d$ and the multiplicities $m_i$ and $k_{I(r)}$ are non-negative.
We notice that for any effective divisor $D$, the strict transform $\tilde{D}$  has non-negative coefficients.
On the other hand, non-effective divisors  behave differently. More precisely, if $D$ is non-effective, the strict transform $\tilde{D}$ may have both non-negative and positive coefficients. 
\end{remark}

\begin{proof}[Proof of  Theorem \ref{monster for empty}]

As the case $n=2$ was worked out in Section \ref{section planar case}.
For $n\ge 3$ by Remark \ref{(I)(II)(III) verified for toric chambers} and  Theorem  \ref{higher cohomologies} we can reduce the proof of Theorem \ref{monster for empty} for a divisor $D$ to the proof of vanishing theorems for its strict transform $\tilde{D}$.

We will prove the vanishing theorems for $\tilde{D}$ by induction on $n$, 
based on the planar case.

Without loss of generality we can assume $m_1\ge \cdots\ge m_{n+2}\ge1$.

We split the proof in two parts corresponding to the two cases $m_1=d+1$ and $m_1\leq d$.

Case (1) $m_1=d+1$.  
The proof is by induction on $b$  and $n$. 
If $n=2$ the statement is proved in Section \ref{section planar case}, while  if  $b=-nd$, i.e. if $m_i=0$ for all 
$i=1,\dots,s$, the statement is trivial.

For the pair $(n,b)$, with $n\ge 3$, and $m_1\ge 1$, we will 
assume the statement to be true for $(n-1,b)$ 
and $(n,b-1)$. We can conclude using Proposition \ref{vanishing r<n+2 empty}
that provides the induction step.

Case (2) $m_1\leq d$. 
 We claim first that $\tilde{D}$ has always non-negative coefficients except in the family of examples contained in Example \ref{example}.

\begin{example}\label{example} 
Let $D:=dH-\sum_{i=1}^{n-1}dE_{i}-m_nE_{n}-m_{n+1}E_{n+1}-m_{n+2}E_{n+2}$ with  $m_n+m_{n+1}+m_{n+2}=d+1$. Notice that only hyperplanes ${H_I}$, with $I=\{1,\ldots, n-1, n\}$, $I=\{1,\ldots, n-1, n+1\}$ or $I=\{1,\ldots, n-1, n+1\}$, split off $D$ with  multiplicity $m_{n}$, $m_{n+1}$ and $m_{n+2}$ respectively. This implies that $h^i(\tilde{D})=0$ for all $i\geq 0$ where $\tilde{D}=-H+E_1+\ldots+E_{n-1}+\sum_{\substack{I\subset \{1,\ldots,n-1\}\\ 2\leq |I|\leq n-1}}E_I$.
\end{example}

To prove this claim, we will first show that conditions $b\leq 1$ and $m_i\leq d$ imply that only hyperplanes ${H}_{I}$ with $I=\{1,\ldots, \check{i}, \ldots, \check{n+2}\}$ or  $I=\{1,\ldots, n-1, n+2\}$ can be fixed part of $D$. Indeed, assume by contradiction that for some $I=\{1,\ldots, \check{i}, \ldots, \check{j}, \ldots n+2\}$ with $j\leq n+1$ and $\{i,j\}\neq \{n,n+1\}$, the hyperplane ${H}_I$ is in the base locus of $D$, i.e. $k_{I}\geq 1$. This implies that the line spanned by the points of $J=\{i,j\}$ is not fixed part of $D$ since $K_{J}\leq 0$ for $I \coprod J=\{1,\ldots,n+2\}$. By the fact that multiplicities are arranged in a decreasing order we obtain
$$d\geq m_{i}+m_{j} \geq m_{i+1}+m_{j+1}.$$
This contradicts the formal definition of base locus of non-effective divisors, see  \eqref{strict transform empty} and \eqref{JJ}, since $K_{I}\geq 1$ and $K_{J}\leq 0$, for $\{i+1,j+1\}\subset I$.

To prove that $\tilde{D}$ has non-negative coefficients except in the cases discussed in Example \ref{example} we consider the divisor $D'=D+E_{n+2}$. We notice that $D'$ is an effective divisor since $m_i\leq d$ and $b(D')=b(D)-1\leq 0$, therefore $\tilde{D'}$ has non-negative coefficients. It is easy to see that $k_I(D')=k_I(D)$ whenever $I=\{1,\ldots, \check{i}, \ldots, n+1,\check{n+2}\}$, while $k_I(D')=k_I(D)-1$ if $I=\{1,\ldots, n-1,\check{n}, \check{n+1}, n+2\}$. Conclude that 
\begin{equation}\label{abc}
\tilde{D}=\tilde{D'}-\tilde{H}_{1, \ldots, n-1,\check{n}, \check{n+1}, n+2}-\sum_{\substack{I\in\JJ^>,n+2\in I\\2\le |I|\le n-1}} E_{I}.
\end{equation}
Also observe that
$$K_{1,\ldots, n-1,\check{n}, \check{n+1}, n+2}(D)=m_1+\cdots+m_{n-1}+m_{n+2}-(n-1)d\leq m_{n+2}.$$
Equality occurs if and only if $m_1=\cdots=m_{n-1}=d$, i.e. $D$ is a divisor as in Example \ref{example}.
Moreover, $K_{1,\ldots, n-1,\check{n}, \check{n+1}, n+2}(D)=m_{n+2}-1$ if and only if $m_1=\cdots=m_{n-2}=d$ and $m_{n-1}=d-1$. We leave it to the reader to check that $\tilde{D}$ has positive coefficients as well. Therefore we conclude
\begin{equation}\label{ab}
K_{1,\ldots, n-1,\check{n}, \check{n+1}, n+2}(D')+1=K_{1,\ldots, n-1,\check{n}, \check{n+1}, n+2}(D)\leq m_{n+2}-2.
\end{equation}
Now $\tilde{D'}$ is an effective divisor with $m_{n+2}'=m_{n+2}-1-K_{1,\ldots,n-1 \check{n}, \check{n+1}, n+2}(D')\geq 2$.  We conclude that the coefficient of the hyperplane class $H$, which is the degree of $D'$,  is at least $2$,  therefore equation \eqref{abc} implies $\tilde{D}$ has positive coefficients. 

\medskip 

We are going to prove now that vanishing theorems hold for divisors $\tilde{D}$, under the hypotheses \eqref{bounds for multiplicities 2}.

 The idea of the proof is as follows. Starting from a divisor, that is of the form
 $\tilde{D}$ defined in \eqref{strict transform empty}, using sequences of type \eqref{seq C}
we decrease by one degree and $n$ multiplicities by passing from $\tilde{D}$ to the divisor in the
 kernel of the sequence. In general, the kernel divisor will not be a strict transform of the form 
\eqref{strict transform empty}, since it can acquire simple linear base locus.
We further use sequences of type \eqref{seq B}  to eliminate the simple base locus and take the strict transform of the kernel divisor. See Section \ref{sequences} for a definition of sequences of type \eqref{seq B} and \eqref{seq C}.

The proof is based on induction on the dimension, $n$ (as the
restricted divisor lives in a an $(n-1)-$dimensional space)
 and on induction on the degree $d$ and the multiplicities $m_{i}$ (as the kernel divisor
has lower coefficients).
\medskip

We recall that $b=K_{I}+K_{I^c}\leq 1$ for any subset $I\subset\{1,\dots,n+2\}$, where $I^c:=\{1,\dots,n+2\}\setminus I$.  Therefore, $K_{I}$ can be both positive or zero for $|I|=2$. 
Moreover, we recall that the multiplicities have been arranged in decreasing order from the beginning. 
Furthermore, since $m_{i}\leq d+1$ and $b\leq 1$, we can have at most $n-1$ multiplicities equal to $d+1$. 
Let us assume, without loss of generality, that $J=\{1,\dots,n\}$.  Fix the hyperplane spanned by the points parametrised by $J$, $H_J$,  and let $\tilde{H}_{J}$  denote its strict transform, cfr. Remark \ref{blow up of hypersurface}. We use first a sequence of type \eqref{seq C} for the divisor $\tilde{D}$:
$$0\rightarrow F:=\tilde{D}-\tilde{H}_{J}\rightarrow \tilde{D} \rightarrow \tilde{D}|_{\tilde{H}_{J}}\rightarrow 0.$$
An easy computation shows that the kernel $F$
could have simple linear base locus, i.e. $K_{I}(F)=1$, for some index sets $I$ with $\{n+1, n+2\}\subset I$.
Using the notation \eqref{IIj}, let  $\JJ(n-1)^>_{n+1,n+2}$ be the set of index sets parametrising these cycles. 
We will use restriction sequences of type \eqref{seq B} to remove the simple linear base locus of $F$ of dimension at most $n-1$, in increasing dimension,
starting from the exceptional divisor of the line spanned by the last two points,
$$0\rightarrow F-E_{n+1,n+2}\rightarrow F \rightarrow F{|_{E_{n+1,n+2}}}\rightarrow 0,$$
and then continuing with the cycles of the set $\JJ(n-1)^>_{n+1,n+2}$.

An algorithm that does this can be constructed following the same idea as in the proofs of Proposition \ref{vanishing cremona}  and \ref{sequence B}.
Precisely, let $\prec$ be the total order on the index sets of $\JJ(n-1)^>_{n+1,n+2}$ inherited from the lexicographical order on the set of all index sets of $\{1,\dots,n+2\}$. For $2\le r\le n-1$, write $s_r$ for the cardinality of the set of index sets of $\JJ(n-1)^>_{n+1,n+2}$ of length $r+1$. We recursively define:
\begin{equation*}
\begin{split}
F(0,0)=&\ F-E_{n+1,n+2},\\
F(r,0)=&\ F(r-1,s_{r-1})-E_{I(r)_0}, \ 2\le \rho\le n-1,\\
F(r,j)=&\ F(r,j-1)-E_{I(r)_j}, \ 1\le j\le s_r.
\end{split}
\end{equation*}
The output of the algorithm is the divisor
$$F(n-1,s_{n-1})=\tilde{F}:=F-\sum_{I\in \JJ(n-1)^>_{n+1,n+2} }E_{I}.$$
As in the proof of Proposition \ref{sequence B}, the fist factor of every restriction is 
$$-\Cr_{r}(h)\boxtimes *,$$
that has vanishing cohomologies. This implies that  $h^i(\tilde{F})=h^i(F)$.
Notice that $b(\tilde{F})=b(F)=b(F)=b\leq 1$, so the kernel $\tilde{F}$ is now in the starting assumption \eqref{bounds for multiplicities 2}. The induction argument on the degree and multiplicities applies to the 
kernel divisor of \eqref{seq C}. Eventually the kernel divisor will become toric 
and by using Theorem \ref{toric theorem}, we know it has vanishing cohomologies.

\medskip

In order to conclude the proof, we analyse the restricted divisor $\tilde{D}|_{\tilde{H}_{J}}$ of the sequence of type \eqref{seq C}
and prove by induction on  $n$ that it has vanishing theorems. 
We denote by $e'$ the trace of the cycle $E_{n+1, n+2}$ on the hyperplane $H_{J}$, $e':=E_{n+1, n+2}|_{\tilde{H}_{J}}$. We have 
$$
\tilde{D}|_{\tilde{H}_{J}}=dh-\sum_{i=1}^{n} m_{i}e_{i}-k_{n+1,n+2}e'- 
\sum_{\substack{k_{I}>0\\ |I|\leq n-2, I\subset J}}k_{I}e_{I}-\sum_{\substack {k_{M}>0\\ |M|=n-1, M\subset J}} k_{M} \tilde{h}_{M},
$$ 
where for every subset $M\subset J$, $|M|=n-1$ then $\tilde{h}_{M}$ is the restriction $E_{M}|_{\tilde{H}_{J}}$ and $E_{M}$ is the exceptional divisor of a codimension$-2$ cycle in the base locus of $D$. Note that $\tilde{h}_{M}$ is (the strict transform of) a the hyperplane of $H_J$ passing through the points of $J$ parametrised by $M$, as in Remark \ref{hyperplane}.
We denote by $G$ the divisor 
$$G:=dh-\sum_{i=1}^{n} m_{i}e_{i}-k_{n+1,n+2}e'.$$
We leave to the reader to check that the restriction is a strict transform
$$\tilde{D}|_{\tilde{H}_{J}}=\tilde{G}.$$

If $k_{n+1, n+2}=0$, the trace $\tilde{D}|_{\tilde{H}_{J}}$ is toric. In this case $b(G)=K_{J}\leq K_{J}+K_{n+1,n+2}=b(D)\leq 1$ and Theorem \ref{toric theorem} applies to the divisor $G$, proving the vanishing theorems for $\tilde{G}$.

 If $k_{n+1,n+2}\geq 1$, the restriction $\tilde{D}|_{\tilde{H}_{J}}=\tilde{G}$ satisfies inequalities \eqref{bounds for multiplicities 2} with $b(F)=b(D)\leq 1$. Since $\tilde{G}$ is a divisor in the blown-up $\PP^{n-1}$ at $s=n+1$ points, we can prove by  induction on $n$ that the vanishing theorems hold for $\tilde{G}$. Proposition \ref{n=2} in fact treats the first step of the induction on $n$, namely the case $n=2$.
We use induction on degree and multiplicities to conclude the vanishing of $\tilde{D}$.
\end{proof}

\subsection{Case $s\ge n+3$}\label{section case n+3}
In this section we prove the result for the case of arbitrary number of points satisfying $b\le \min(n-s(d),s-n-2)$. 
\begin{proof}[Proof of Theorem \ref{empty arbitrary points}] 
As in  Theorem \ref{monster for empty} we observe that 
using Remark \ref{(I)(II)(III) verified for toric chambers} and  Theorem  \ref{higher cohomologies} we reduce the proof of Theorem \ref{monster for empty} for a divisor $D$ to the proof of vanishing theorems for its strict transform $\tilde{D}$.

We assume $m_1\ge \cdots \ge m_s$.
The effective case, namely  $m_i\leq d$, was proved in Theorem \ref{vanishing for >=n+3}. We can assume $m_1=d+1$.
Notice that $K_{2,\dots,n+1}=b(D)-m_1-
\sum_{i\ge n+2}m_i+d\le s-n-2-m_1-(s-n-1)+d\le 0$, 
the inequality follows from  $b\le s-n-2$, $m_i\ge 1$ and $m_1=d+1$.
This implies that the hyperplane 
spanned by $\{2,\dots, n+1\}$ is not contained in the base locus of $D$.  Furthermore, no
hyperplane spanned by $I(n-1)\subset\{2,\dots,s\}$ is. 

Notice  that the inequality $b\le s-n-2$ 
implies $K_{I}+K_{J}=\sum_{i\in I\cup J}m_{i}-nd\leq 0$ for any disjoint
 sets $I,J$ with $|I|+|J|=n+2$. Therefore, $\II$, the set of all multi-indices in $\{1,\dots,s\}$,
 satisfies conditions (I) and (III) of Section \ref{blow-up construction}.

We conclude that $h^i(\tilde{D})=0$ using induction on $n$ (the case $n=2$ was 
covered in proposition \ref{n=2}), and the induction procedure of Section \ref{section induction procedure}.

\end{proof}

\section{Vanishing theorems for points in star configuration in $\PP^n$}\label{star configuration}

As an application of the results proved in the previous sections, we 
compute the  number of global sections and  
 prove  vanishing
 theorems for the cohomology groups of the strict transforms along the linear
 base locus of some families of divisors interpolating points in \emph{star configuration} in $\PP^n$.

A star configuration of points is a collection of points satisfying some 
particular geometric relation. 
They have been object of study in many
 papers lately, see \cite{Dum2, geramita} and references therein.

Given $l$ hyperplanes in $\PP^n$ that meet properly, i.e. not three of them intersecting 
along a $\PP^{n-2}$, not four of them intersecting along a
 $\PP^{n-3}$ etc., 
a \emph{star configuration} of \emph{dimension} $r$ \emph{subspaces}
is the set given by the ${l\choose n-r}$ linear
 subspaces of dimension $r$ in  $\PP^n$ formed by taking all 
possible intersections of $n-r$ among the $l$ hyperplanes.

In \cite[Theorem 3.2]{geramita}, the authors  compute the Hilbert function of the ideals
of star configurations of dimension $r$  subspaces of multiplicity two.
This provides a complete classification  of linear systems in $\PP^n$ interpolating
such a scheme.
In Theorem \ref{star} we compute the number of global sections of a class of linear systems in $\PP^n$ 
interpolating star configurations of points obtained by $l=n+2$ hyperplanes with higher multiplicities.

\begin{remark}\label{star as hyperplane section}
When $l=n+2$,  star configurations of  dimension $r$
subspaces in $\PP^n$ are obtained as follows. 
Let us embed $\PP^n\hookrightarrow H\subset\PP^{n+1}$ and denote by $p_1,\dots,
p_{n+2}$ a general collection of points of $\PP^{n+1}$,
 that we may think of as the coordinate points, with respect to which $H$ is a general
hyperplane.
The family of points 
$q_{ij}:=L_{ij}\cap H\in H$, for all lines $L_{ij}=\langle p_i,p_j\rangle\subset\PP^{n+1}$, forms 
a star configuration of points in $H$. Indeed,  denoting by $H_l$ the hyperplane spanned by 
all $p_i$'s with $i\ne l$ for all $1\le l\le n+2$, we can write 
$q_{ij}=H\cap\bigcap_{l\ne i,j}H_{l}$.
Similarly, the family of $r$-linear subspaces
$\lambda_I:=L_I\cap H$,
 for all multi-indices $I=\{i_1,\dots,i_{r+2}\}\subset\{1,\dots, n+2\}$, is
a  star configuration of dimension $r$ subspaces in $H$.
\end{remark}

We now study effective divisors in $\PP^n$ interpolating star configurations of points.
Set $Y_{(0)}=Y^n_{(0)}$ to be the blow-up of $\PP^n$ at the star configuration
of points given as intersections of ${{n+2}\choose n}$ hyperplanes. Adopting the same notation of
Remark \ref{star as hyperplane section}, we call $q_{ij}$, $1\le i<j\le n+2$ such points.
Let us denote 
 by $h$ the hyperplane class and by $e_{ij}$ the exceptional divisors.

\begin{remark}
As in Remark \ref{star as hyperplane section}, let us embed 
$\PP^n\hookrightarrow \PP^{n+1}$ and, as in  Section \ref{blow-up construction},  
let  $X_{(1)}=X^{n+1}_{(1)}$ denote
the blow-up of $\PP^{n+1}$ at  general points $p_1,\dots, p_{n+2}$ and,
subsequently, along the lines spanned by those points, with exceptional divisors $E_{ij}$,
$1\le i<j\le n+1$. By writing $h=H{|_H}$ and $e_{ij}={E_{ij}}{|_H}$, 
we obtain the following isomorphism $Y_{(0)}\cong {X_{(1)}}{|_H}$.
\end{remark}

Given integers $d\ge m_1,\dots, m_{n+2}\ge0$, we use the notation $k_{ij}=\max(m_i+m_j-d,0)$.  Assume  
\begin{equation}\label{condition star}
\sum_{i=1}^{n+2} m_i\le (n+1)d,
\end{equation}
and consider the following divisor on $Y_{(0)}$:
$$
\Delta:=dh-\sum_{1\le i<j\le n+2}k_{ij}e_{ij}.
$$

We prove that it is only linearly obstructed, with linear base locus supported on 
the star configurations of linear subspaces $\lambda_I$ defined above
 (Remark \ref{star as hyperplane section}) and 
that its subsequent strict transforms after blowing-up the dimension $r$ star configuration have vanishing 
theorems. 

Let $\II$ be the set of all multi-indices in $\{1,\dots,n+2\}$, and for each
multi-index $I(r)\subset\II$ of cardinality $r+1$, we use the notation \eqref{mult k}.
For increasing $r$, let $Y_{(r)}=Y^{n}_{(r)}$ denote the blow-up of $Y_{(r-1)}$ along the strict transform of the
 star configuration of $r$-subspaces $\lambda_{I(r)}$ in $H$ and let 
$\Delta_{(r)}$  be the strict transform of $\Delta$.

\begin{theorem}\label{star}
In the above notation,  assume that \eqref{condition star} is satisfied. then we have
$$
h^{0}( \Delta)={{n+d}\choose{n}}+
      \sum_{I(r)\in\II,  r\ge 1}(-1)^{r}{{n+k_{I(r)}-r}\choose{n}}.
$$
Moreover, 
for all $1\le r\le n-1$, 
$h^i( \Delta_{(r)})=0$, for all $i\ne 0,r+1$, and
$$
h^{r+1}( \Delta_{(r)})=\sum_{I(\rho)\in\II, \rho\ge r+2}(-1)^\rho{{n+k_{I(\rho)}-\rho}\choose{n}}.
$$
In particular, $h^i( \tilde{\Delta})=0$, for all $i> 0$.
\end{theorem}

\begin{proof}

Recall the inclusion $H\subset\PP^{n+1}$ and let $X_{(r+1)}:=X^{n+1}_{(r+1)}$ 
 be the blow-up of $\PP^{n+1}$ first at the points $p_1,\dots, p_{n+2}$ and then along 
the linear subspaces $L_I$ of dimension bounded above by $r+1$, in increasing dimension.
We have the following inclusion
$Y_{(r)}\subset X_{(r+1)}$.

Let us consider the following
 divisors  on $X_{(0)}$  
$$
D=dH-\sum_{i=2}^{n+2}m_i E_i, \quad 
 D'=D-H.
$$
Notice that, since \eqref{condition star} is satisfied, then $D$ is effective (cfr. \eqref{toric effective cone}) and $D'$ is either effective or satisfies condition \eqref{n+1 b pos}.

Abusing notation, denote by $H$  the strict transform of $H\cong\PP^n\subset\PP^{n+1}$ and 
notice that the restriction ${D_{(r+1)}}{|_H}$ belongs to the linear system associated with $\Delta_{(r)}$:
$$
{D_{(r+1)}}{|_H}\in|\Delta_{(r)}|.
$$ 
Consider the following restriction sequence 
\begin{equation}\label{restriction to H}
0\to D_{(r+1)}-H\to D_{(r+1)}\to {D_{(r+1)}}{|_H}\to 0.
\end{equation}

We claim that the strict transforms of the linear $\rho$-cycles
contained in base locus of $D_{(r+1)}-H$ have multiplicity bounded above by $\rho$, for all 
 $\rho\le r+1$.
 Hence, by Theorem \ref{higher cohomologies} (2), they do 
not provide linear obstruction to such a divisor, namely $h^i(D_{(r+1)}-H)=h^i(D'_{(r+1)})=0$,  
for all $i\ne 0,r+2$.
Moreover the following holds: $h^0(D_{(r-1)}-H)=h^0(D'_{(r-1)})=h^0(D')$ and $h^0(D_{(r-1)})=h^0(D)$.
By Theorem \ref{monster for effective} we have that $h^i(D_{(r+1)})=0$, for all $i\ne 0,r+2$.
Therefore, we  conclude that
\begin{align*}h^0(\Delta_{(r)})&=h^0(D)-h^0(D'),\\ 
h^{r+1}(\Delta_{(r)})&=h^{r+2}(D'_{(r+1)})-h^{r+2}(D_{(r+1)}),
\end{align*}
and
that all higher cohomology groups vanish, by means of the  long exact sequence in cohomology associated with \eqref{restriction to H}.

In order to conclude, we are left to prove the claim.  Notice that the multiplicity of containment of $L_{I(\rho)}$ in the base locus of $D'$ is $k'_{I(\rho)}:=\max(\sum_{i\in I(\rho)}m_i-\rho(d-1),0)$. Moreover, in $D_{(r+1)}-H$ the strict transform of the exceptional divisor of $L_{I(\rho)}$ has been removed
$k_{I(\rho)}:=\max(\sum_{i\in I(\rho)}m_i-\rho d,0)$ times. To conclude, it is enough to observe that $k'_{I(\rho)}-k_{I(\rho)}\le\rho$.
\end{proof}

\begin{corollary}
The linear subspace ${\lambda_I}$ is contained with multiplicity
$k_I$ in the base locus of $\Delta_{(r)}$.
\end{corollary}
\begin{proof}
Let $D$  be as  in the proof of Theorem \ref{star}.
The linear base locus of $\Delta_{(r)}$
is the intersection with the hyperplane $H$ of the linear base locus of $D_{(r+1)}$, described in 
Proposition \ref{sharpBLL},
and in particular it is supported
at the dimension $\rho$ star configurations, with $\rho\ge r+1$.
\end{proof}

An interpretation of the above corollary is that the only obstructions
are the linear subspaces $\lambda_I$.

Moreover, this suggests a definition of virtual linear dimension for divisors interpolating points  in star configuration in $\PP^n$, that generalises the notion of virtual linear dimension
for divisors interpolating points  in general position  that was introduced in
\cite{BraDumPos} and that has been extensively studied throughout  this paper.
While in the general case the linear obstructions are  
the linear subspaces spanned by the points, in this case they 
are given by the star configurations of linear subspaces.

\begin{remark}
One may study general hyperplane
sections of effective linearly obstructed divisors $D$ in $\PP^{n+1}$ interpolating arbitrary
numbers of points in general position with bounds \eqref{EffectivityCondition} or \eqref{bounds for multiplicities 2}. 
Using Theorem \ref{vanishing for >=n+3} or Theorem \ref{monster for effective} one analyses the divisor $D$ and using Corollary \ref{empty arbitrary points} or 
Theorem \ref{monster for empty} respectively one analyses the kernel divisor.
 The resulting restricted divisor
$\Delta={D_{(1)}}{|_H}$ in $Y^n_{(0)}$ interpolates points in special configuration
and is linearly obstructed. 
Moreover, a cohomological description 
such as the one established in Theorem \ref{star} can  be obtained for such divisors.
\end{remark}

\end{document}